%% file: SLNetwork_v29.tex
\RequirePackage{fix-cm}
\documentclass[smallextended]{svjour3}       
\smartqed  
\usepackage{graphicx}
\usepackage{amsmath}
\usepackage{amssymb}
  \usepackage{graphics} 
  \usepackage{epsfig} 
 \usepackage{lipsum}
\usepackage{mathtools}
 \usepackage{hyperref}
\def\a{\alpha}

\def\e{\varepsilon}

\newcommand{\ox}{{\overline x}}
\newcommand{\oy}{{\overline y}}
\newcommand{\oot}{{\overline t}}
\newcommand{\os}{{\overline s}}

\newcommand{\at}{{\widetilde \a}}

\newcommand{\R}{{\mathbb{R}}}

\newcommand{\E}{{\mathbb E}}

\newcommand{\Dt}{{\Delta t}}
\newcommand{\Dx}{{\Delta x}}

\input mymacro
\begin{document}

\title{A semi-Lagrangian scheme for Hamilton-Jacobi equations on networks with application to traffic flow models. 
}

\titlerunning{A semi-Lagrangian scheme for HJ equations on networks}        

\author{Elisabetta Carlini        \and
        Adriano Festa \and Nicolas Forcadel
}


\institute{E. Carlini, \at
              Sapienza University of Rome, 00185, Roma, Italy.\\
              \email{carlini@mat.uniroma1.it}          
           \and
           A. Festa, \at
              INSA Rouen, 76800 Saint-\'Etienne-du-Rouvray, France.\\
\email{adriano.festa@insa-rouen.fr}
\and
N. Forcadel, \at INSA Rouen, 76800 Saint-\'Etienne-du-Rouvray, France.\\
\email{nicolas.forcadel@insa-rouen.fr}
}

\date{Received: date / Accepted: date}

\maketitle

\begin{abstract}
We present a semi-Lagrangian  scheme for the approximation of a class of Hamilton-Jacobi-Bellman equations on networks. The scheme is explicit and stable under some technical conditions. We prove a convergence theorem  and  some error estimates. Additionally, the theoretical results are validated by numerical tests. Finally,  we apply the scheme to simulate traffic flows modeling problems. 
\end{abstract}

\section{Introduction}

The attention to the study of linear and nonlinear partial differential equations on networks raised consistently in the last decades motivated by the extensive use of systems like roads, pipelines, and electronic and information networks. 

In particular, extensive literature has been developed for vehicular traffic systems modeled through conservation laws. Existence results can be found in  \cite{garavello2006traffic}, and some partial uniqueness results (for a limited number of intersecting roads) in \cite{garavello2007conservation,andreianov2011theory}. Nonetheless, the lack of uniqueness on the junction point obliges to add some additional conditions that may be ambiguous or difficult to derive.  
More recently, another kinds of macroscopic models appears. These models rely on the Moskowitz function and make appear an Hamilton-Jacobi equation (see \cite{newell}).

The theory of Hamilton-Jacobi (HJ) equations on networks is very recent.  It is difficult to extend the classic framework to the network context because these equations do not have in general regular solutions and the notion of weak solution (\emph{viscosity solution}) must be adapted to preserve some properties on the junction points. An additional difficulty comes from possible discontinuities on data of the problem. Some theoretical results are contained in the early works \cite{achdou2013hamilton,camilli2018flame,imbert2013flux,imbert2013hamilton,schieborn2013viscosity} where, using some appropriate definitions of weak solutions, the authors prove the well-posedness of the problem. We also refer to the works \cite{barles2016flux,LionsSouganidis} for simplified proof of uniqueness. Concerning numerical schemes for this kind of equations, there is very few theoretical results. Let us mention the finite differences scheme proposed in \cite{camilli2018flame,costeseque2015convergent} and the paper \cite{imbert2015error} in which they prove some error estimates for this scheme.

In this paper, we adopt the notion of solutions as introduced in \cite{imbert2013flux}, which has some good advantages in term of generality, and we introduce a new numerical scheme for HJ equations on a network. For the sake of simplicity, we consider a simplified network (\emph{a junction}), but the result can be extended to a more general class of problems, including more complex structures.

We propose a semi-Lagrangian (SL) scheme by discretizing the \emph{dynamic programming principle} presented in \cite{imbert2013flux}. This scheme generalizes what introduced in \cite{camilli2013approximation}, and it enables discrete characteristics to cross the junctions. This property makes the scheme absolutely stable, allowing large time steps, and it is the main advantage compared to finite differences and finite elements schemes. We prove consistency and monotonicity which imply the convergence of the scheme. 

We also derive some consistency errors for the numerical solution that we can obtain in two different cases: for state independent Hamiltonians, where controls are constant along the arcs, and in a more general scenario. In the simplified case, we obtain a first order convergence estimate. In the second case, the key result is a consistency estimate that leads to an \emph{a priori} error estimate. The proof is obtained combining some techniques derived from the papers on regional optimal control problems    \cite{barles2013bellman,barles2016flux}.

\paragraph{{\bf Structure of the paper:}} in Section \ref{Sect:basics} we recall some basic notions for  junction networks, the definition of \emph{flux-limited viscosity solutions}  and the relation with optimal control problems on networks. In Section \ref{Sect:scheme}, we derive the SL scheme and prove its basic properties: consistency, monotonicity, and convergence. In Section \ref{sect:bounds},  we present the main result (Theorem \ref{teo:bounds2}) concerning the error estimate. In Section \ref{Sect:traffic},  we discuss the connection  between HJ equations and traffic flow model. Finally, in Section \ref{Sect:tests}, we show through numerical simulations the efficiency and the accuracy of our new method.

\section{Hamilton-Jacobi equations on networks} \label{Sect:basics}

A network is a domain composed of a finite number of nodes connected by a finite number of edges. To simplify the description of such system we focus on the case of a \emph{junction}, which is a network composed of one node and a finite number of edges. We follow \cite{imbert2013flux} and the notations therein to describe the problem.\\
Given a positive number $N$, a junction $J$ is a network of $N$ half lines $J_i:=\{k\,e_i, k\in\R^+\}$ (where each line is isometric to $[0,+\infty)$ and $e_i$ is a unitary vector centered in $0$) connected in a \emph{junction point} that we conventionally place at the origin. We then have  
$$ J:=\bigcup_{i=1,...,N}J_i, \quad J_i\cap J_j =\{0\}, \quad\forall i\neq j,\quad i,j\in\{1,...,N\}.$$
\begin{figure}[h!]
\begin{center}
\includegraphics[height=4.5cm]{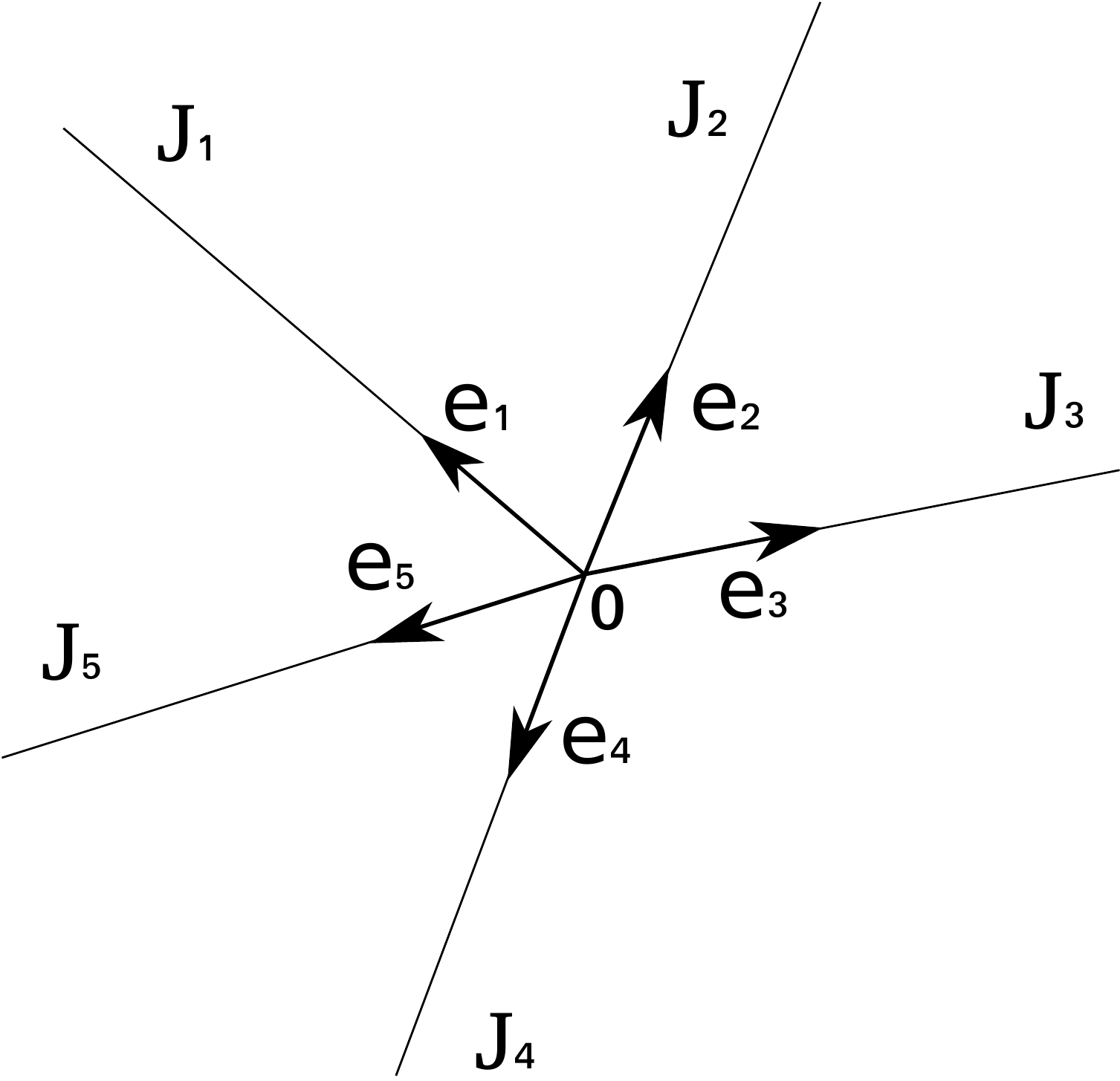} 
\caption{Junction with $N=5$ edges.} \label{figJun}
\end{center}
\end{figure}

We consider the geodesic distance function on $J$ given by
$$d(x,y)=\left\{\begin{array}{lll}
|x-y|, & & \hbox{if $x,y\in J_i$ for one $i\in\{1,...,N\}$},\\
|x|+|y|, & & \hbox{otherwise.}
\end{array} \right.$$
For a real-valued function $u$ defined on $J$, $\partial_i u(x)$ denotes the (spatial) derivative of $u$ at $x \in J_i$
and the gradient of $u$ is defined as:
\begin{equation}
  u_x:=\left\{
\begin{array}{lll}
 \partial_i u(x) & & \hbox{if }x\in J_i\setminus \{0\},\\
 \left(\partial_1 u(x),\partial_2 u(x),...,\partial_N u(x)\right) & &\hbox{if }x=0.
\end{array}
\right.
\end{equation}
We can now describe  our problem. Consider the following evolutive HJ equation on the network $J$
\begin{equation}\label{eq:hjnet}
\left\{
\begin{array}{ll}
 \partial_t u(t,x)+H_i(x,u_x(t,x))=0 & \hbox{ in } (0,T) \times J_i\setminus \{0\},\\
\partial_t u(t,x)+F_A(u_x(t,x))=0 & \hbox{ in }  (0,T) \times \{0\},
\end{array}
\right.
\end{equation}
with the initial condition 
\begin{equation}\label{eq:initial}
u(0,x)=u_0(x)\quad \hbox{ for }x\in J,
\end{equation}
where $u_0(x)$ is globally Lipschitz continuous on $J$. We suppose that standard assumptions on the Hamiltonian $H$ (cf. i.e. \cite{BardiCapuz@incollection}) 
hold:
\begin{itemize}
\item[(\bf{H1})](Regularity) for all $L>0$ there exists a modulus of continuity $\omega_L$ such that for all $|p|,|q|\leq L$ and $x\in J_i$ 
$$|H_i(x,p)-H_i(x,q)|\leq \omega_L(|p-q|);$$ in addition, $H(\cdot,p)$ is Lipschitz continuous w.r.t. the space variable.
\item[ (\bf{H2})](Uniform coercivity) $H_i(x,p)\rightarrow +\infty$ for $|p|\rightarrow +\infty$ uniformly for every $x\in J_i\cup \{+\infty\}$, $i=1,...,N$;
\item[ (\bf{H3})](Convexity) $\{H_i(x,\cdot) \leq \lambda \}$ is convex for every choice of $\lambda\in\R$ and a fixed $x\in J$.
\end{itemize}
Using the convexity hypothesis and the coercivity, there exists a $\hat p_i$ such that the Hamiltonian $H_i$ is non-increasing in $(-\infty,\hat p_i]$ and non-decreasing in  $[\hat p_i,\infty)$. We introduce the (respectively) non-increasing and non-decreasing functions 
\begin{equation*}
H^-_i(x,p):=\left\{
\begin{array}{ll}
 H_i(x,p)& \hbox{for }p\leq \hat p_i \\
 H_i(x,\hat p_i)& \hbox{for }p > \hat p_i 
\end{array}
\right.
\hbox{ and } 
H^+_i(x,p):=\left\{
\begin{array}{ll}
 H_i(x,p)& \hbox{for }p\geq \hat p_i \\
 H_i(x,\hat p_i)& \hbox{for }p < \hat p_i.
\end{array}
\right.
\end{equation*}
 Given a parameter $A\in \R \cup \{-\infty\}$ (flux limiter),  we define the operator $F_A:\R^N\rightarrow \R$ on the junction point as
\begin{equation}
F_A(p):= \max\left(A,\max_{i=1,...,N}H^-_i(0,p_i)\right).
\end{equation}
In order to introduce the notion of viscosity solution, we introduce the class of test functions. For $T>0$, set $J_T=(0,T)\times J$.
We define the class of test functions on $J_T$  and on $J$ as
$$C^k(J_T)= \{\varphi \in C(J_T),\,  \forall i=1,\dots,N,\, \varphi \in C^k((0,T)\times J_i)\},$$
$$C^k(J)=\{\varphi \in C(J), \, \forall i=1,\dots,N,\, \varphi \in C^k(J_i)\}.$$
We recall also the definition of upper and lower semi-continuous envelopes $u^*$ and $u_*$ of a (locally bounded) function $u$ defined on $[0, T ) \times J$,
$$u^*(t, x) = \limsup_{(s,y)\rightarrow(t,x)}  u(s, y) \quad  \hbox{ and }\quad  u_*(t, x) = \liminf _{(s,y)\rightarrow(t,x)} u(s, y).$$
We  say that a test function $\varphi$ touches a function $u$ from below (respectively from above) at $(t, x)$ if $u-\varphi$ reaches a minimum (respectively maximum) at $(t, x)$ in a neighborhood of it.

\begin{definition}[Flux-limited solutions] Assume that the Hamiltonians satisfy ({\bf H1})-({\bf H3}) and let $u :[0, T ) \times J \rightarrow \R$.
\begin{itemize}
\item[i)] We say that $u$ is a flux-limited sub-solution (resp. flux-limited super-so\-lu\-tion) of \eqref{eq:hjnet} in $(0, T ) \times J$ if for all test function $\varphi\in C^1(J_T)$ touching $u^*$  from above (resp. $u_*$ from below) at $(t_0, x_0)\in J_T$, we have
\begin{equation}
\begin{array}{lll}
\varphi_t(t_0 , x_0) + H_i (x_0,\varphi_x(t_0 , x_0)) \leq 0 & \quad \hbox{ (resp. $\geq 0$) } & \hbox{ if $x_0\in J_i$},\\
\varphi_t(t_0 , x_0) + F_A (\varphi_x(t_0 , x_0)) \leq 0 & \quad \hbox{ (resp. $\geq 0$)} \quad  & \hbox{ if $x_0=0$}.
\end{array}
\end{equation}
\item[ii)] We say that $u$ is a flux-limited sub-solution (resp. flux-limited super-so\-lu\-tion) of \eqref{eq:hjnet} on $[0, T ) \times J$ if additionally
\begin{equation}
u^*(0, x) \leq u_0 (x) \quad {\hbox{(resp. }} u_*(0, x) \geq u_0 (x)) {\hbox{ for all}} \; x \in J.
\end{equation}
\item[iii)] We say that $u$ is a flux-limited solution if $u$ is both a flux-limited sub-solution and a flux-limited super-solution.
\end{itemize}
\end{definition}

Thanks to the work of Imbert and Monneau \cite{imbert2013flux}, we have the following result which gives an equivalent definition of viscosity solutions for \eqref{eq:hjnet}. We use this equivalent definition in particular in the definition of the consistency in Section \ref{Sect:scheme}.

\begin{theorem}[Equivalent definition for sub/super-solutions] \label{th:1}
Let $\bar H^0 = \max_j \min_pH_j(p)$ and consider $A \in [\bar H^0, +\infty)$. Given solutions $p_{i}^{A} \in \mathbb{R} $  of
\be
      \label{definitionEquivalenteDefDesP}
      {H}_i\left( p_i^{A} \right)= {H}^+\left( p_i^{A} \right)=A \ee
    let us fix any time independent test function $\phi^0(x)$ satisfying, for $i=1,\dots, N$,
    \begin{eqnarray}
      \nonumber
      \partial_{i}\phi^0 (0)= p_{i}^{A}.
    \end{eqnarray}
    Given a function $u:(0,T)\times J\rightarrow \mathbb{R}$, the following properties hold true.

    i) If $u$ is an upper semi-continuous sub-solution of \eqref{eq:hjnet} with $A=H_0$, for $x\neq 0$, satisfying 
    \begin{eqnarray}
    \label{conditionDefEquivalentePBLimite}
    u(t,0)=\limsup_{(s,y)\rightarrow (t,0),\ y\in R^*_i}u(s,y),
\end{eqnarray}     
then $u$ is a $H_0$-flux limited sub-solution.

    ii) Given ${A}>H_0$ and $t_0\in(0,T)$, if $u$ is an upper semi-continuous sub-solution of \eqref{eq:hjnet} for $x\neq 0$, satisfying \eqref{conditionDefEquivalentePBLimite}, and if for any test function $\varphi$ touching $u$ from above at $(t_0,0)$ with 
    \begin{eqnarray}
      \label{definitionEquivalenteDefDeFonctionTestGlobale}
      \varphi(t,x)= \psi(t) + \phi^0(x),
    \end{eqnarray}
    for some $\psi \in C^2\left((0,+\infty)\right)$, we have
    \begin{eqnarray}
    \nonumber
      \varphi_t + F_{A} \left( \varphi_x \right) \leq 0 \quad \mbox{at }(t_0,0),
    \end{eqnarray}
    then $u$ is a ${A}$-flux limited sub-solution at $(t_0,0)$.

    iii) Given $t_0\in(0,T)$, if $u$ is a lower semi-continuous super-solution of \eqref{eq:hjnet} for $x\neq 0$ and if for any test function $\varphi$ satisfying (\ref{definitionEquivalenteDefDeFonctionTestGlobale}) touching $u$ from above at $(t_0,0)$ we have
    \begin{eqnarray}
    \nonumber
      \varphi_t + F_{A} \left( \varphi_x \right) \geq 0 \quad \mbox{at }(t_0,0),
    \end{eqnarray}
    then $u$ is a ${A}$-flux limited super-solution at $(t_0,0)$.
    \label{thDefinitionEquivalenteSurETSousSolutionsJunction}
\end{theorem}

\subsection{Optimal control interpretation and dynamic programming principles}
We describe a natural application of equations \eqref{eq:hjnet} for a finite-horizon optimal control problem on the network $J$. We recall  several results contained in \cite{imbert2013flux} that are useful in the next sections. \\
Let us define the set of admissible dynamics on the network $J$ connecting the point $(s,y)$ to $(t,x)$ as
\begin{equation}
   \Gamma_{s,y}^{t,x}:=\left\{
   \begin{array}{ll}
    (X(\cdot),\alpha(\cdot))\in {\rm{Lip}}([s,t];J)\times L^\infty ([s,t];\R^{N+1})\\
    \dot X(\tau)=\alpha(\tau), \quad \tau\in[s,t]\\
    X(s)=y, \quad X(t)=x
\end{array}\right\}.
\end{equation}
We denote by $(\alpha_0,\alpha_1,...,\alpha_N)$ the $N+1$ components of the control function $\alpha:(0,T)\rightarrow\R^{N+1}$, where   $\alpha_i(t)$ is the control function defined on  the branch $J_i$   for $i=1,...,N$ and  $\alpha_0(t)$ is the control function defined  on the junction point.\\
We define a \emph{cost function}, 
$$L(x,\alpha):=
\left\{
   \begin{array}{ll}
    L_i(x,\alpha_i) & \hbox{if }x\in J_i,\\
    L_0(\alpha_0) & \hbox{if }x=0,
\end{array}\right. 
$$
 where   for $ i=1,\dots,N$  and we assume the following
 \begin{itemize}
 \item[{\bf (A1)}] $L_i:\R^+\times\R \rightarrow \R$ are \emph{strictly convex} (w.r.t. the second argument) and uniformly Lipschitz continuous functions,
 \item[{\bf (A2)}] $L_i$ are \emph{strongly coercive} w.r.t. the control argument uniformly in $x$
 ($L_i(x,\alpha_i)/|\alpha_i|\rightarrow +\infty$ for $|\alpha_i|\rightarrow+\infty$ uniformly in $x\in \R^+$).
 \end{itemize}
  In addition, $L_0:\R \rightarrow \R$  is defined as
$$ L_0(\alpha_0):=\left\{
   \begin{array}{ll}
    \bar L_0 &\quad  \hbox{if }\alpha_0=0,\\
    +\infty &\quad  \hbox{otherwise,}
\end{array}\right. 
$$
for a given  $\bar L_0\in \R$.
We define the {\em value function} of the optimal control problem  as
\begin{equation}\label{eq:value}
 u(t,x)=\inf_{y\in J} \inf_{(X(\cdot),\alpha(\cdot))\in \Gamma^{t,x}_{0,y}}\left\{u_0(X(0))+\int_0^t L(X(\tau),\alpha(\tau))d\tau\right\}.
\end{equation}
 
It has been proved in \cite{imbert2013flux} that the following \emph{dynamic programming principle} (DPP) holds. 
\begin{proposition} \label{hjdpp}
For all $x\in J$, $t\in (0,T]$, $s\in [0,t)$, the value function $u$ defined in \eqref{eq:value} satisfies
\begin{equation}\label{eq:dyn}
u(t,x)=\inf_{y\in J} \inf_{(X(\cdot),\alpha(\cdot))\in \Gamma^{t,x}_{s,y}}\left\{u(s,X(s))+\int_s^t L(X(\tau),\alpha(\tau))d\tau\right\}.
\end{equation}
\end{proposition}
A direct approximation of the DPP \eqref{eq:dyn} is the basis for the  scheme which we describe  in the next section.\\
The following Theorem   characterizes  the value function  \eqref{eq:value} as the solution of a HJ equation (for the proof see \cite{imbert2013flux}).
\begin{theorem}\label{hjcont}
The value  function $u$ defined in \eqref{eq:value} is the unique viscosity solution of \eqref{eq:hjnet} with 
\begin{equation}\label{hamilt} H_i(x,p):= \sup_{\alpha_i\in \R} \left\{ \alpha_i \,p - L_i(x,\alpha_i)\right\}, \hbox{ and } A=- \bar L_0.
\end{equation}
\end{theorem}

\begin{proposition}\label{Hproperty}
Under  assumption {(\bf{A1})}-{(\bf{A2})}  the following assertions hold true:
\begin{enumerate}
\item[$i)$] for every $x\in\R^+\cup \{+\infty\}$, $\alpha_i\in\arg\sup_{\alpha_i\in \R} \{\alpha_i p -L_i(x,\alpha_i)\}$ is bounded. 
\item [$ii)$] the non increasing part of $H_i(x,p)$ with respect of $p_i$ is given by
$$H^-_i(x,p_i)=\sup_{\alpha_i\leq 0}\{\alpha_i p_i -L_i(x,\alpha_i)\}.$$
\end{enumerate}
furthermore the Hamiltonian \eqref{hamilt} satisfies properties $({\bf{H1}})-({\bf{H3}})$. 
\end{proposition}
\begin{proof}

From assumption  {(\bf{A1})}-{(\bf{A2})}, $\alpha_i p-L_i(x,\alpha_i)$ is a continuous function (negatively) coercive therefore there exists a compact interval $[-\mu,\mu]$, $\mu\in\R$, such that
$$  \sup_{\alpha_i\in \R} \left\{ \alpha_i \,p - L_i(x,\alpha_i)\right\}=\sup_{\alpha_i\in [-\mu,\mu]} \left\{ \alpha_i \,p - L_i(x,\alpha_i)\right\}.$$ Then, $i)$ holds. Assertion $ii)$ follows from Lemma 6.2 in \cite{imbert2013flux}.\\
From $i)$, we have
$$H_i(x,p)-H_i(x,q)\leq \bar \alpha |p-q|$$
where $\bar \alpha$ is the minimizer in $H_i(x,q)$. Exchanging the role of $p,q$, we get ({\bf{H1}}).\\
 Taking $\alpha=1$ in  \eqref{hamilt} we have
$$H_i(x,p)\geq p-L_i(x,1).$$
The same argument for $\alpha=-1$ gives  $H_i(x,p)\rightarrow +\infty$ for $|p|\rightarrow +\infty$, then ({\bf{H2}}) holds.
Finally  ({\bf{H3}}) holds since $H_i$ is  the superior envelope of convex functions. 
\qed
\end{proof}

\section{Numerical resolution: a semi-Lagrangian scheme} \label{Sect:scheme}

Let us introduce a uniform discretization of the network $(0,T)\times J$. The choice of a uniform discretization is not restrictive, and the scheme can be easily  extended to non-uniform grids. Given $\Delta t$ and $\Delta x$ in $\R^+$, we define $\Delta=(\Delta x, \Delta t)$, $N_T=\floor{ T/\Delta t}$ ($\floor{\cdot}$ is the truncation operator) and 
$$ \mathcal G^{\Delta} :=\{t_n: n=0,\dots, N_T  \}\times J^{\Delta x}$$
where
$$J^{\Delta x}:=\bigcup_{i=1,...,N} J_i^{\Delta x}, \quad J_i^{\Delta x}=\{k\Delta x\,e_i:k\in\NN\}.$$
We call $t_{n}=t_n$ for $n=0,\dots, N_T$ and we derive a discrete version of the dynamic programming principle \eqref{eq:hjnet} defined on the grid $ \mathcal G^{\Delta} $.
To do so, as usual in  first-order SL schemes, we discretize the trajectories in $\Gamma^{t_{n+1},x}_{t_{n},y}$ by one step of  Euler scheme.
For $i\in \{1,\dots,N\}$, let $x\in J_i$  and let $\alpha \in \R^{N+1}$ be such that $\alpha_i \Delta t\leq |x|$, then the approximated trajectory gets
$$x\simeq y+\alpha_i\Delta t .$$
 In this case, the discrete backward trajectory $x-\Delta t \alpha_i$ remains on $J_i$, and, by also applying a  rectangle formula, the discrete version of \eqref{eq:dyn} at the point $(t_{n+1},x)$ is
$$u(t_{n+1},x)\simeq u(t_n,x-\alpha_i\Delta t e_i)+\Delta t L_i(x,\alpha_i).$$ 
In opposite case $\alpha_i \Delta t> |x|$, the discrete trajectory has passed through the junction. Denoting  $ s_0\in [0,\Delta t-\frac{|x|}{\alpha_i}]$ the time spent by the trajectory   at the junction point,   $J_j$ the arc from which the trajectory comes and $\hat t:=\left(\Delta t- s_0-\frac{|x|}{\alpha_i}\right)$ the time spent by the trajectory  on the arc $J_j$, the approximation of \eqref{eq:dyn} at the point $(t_{n+1},x)$ becomes
\begin{equation*}
u(t_{n+1},x)\simeq u\left(t_n,-\alpha_j \hat t e_j\right) 
+\hat t\, L_j(0,\alpha_j)+ s_0 L_0(\alpha_0)+\frac{|x|}{\alpha_i}L_i(x,\alpha_i).
\end{equation*}
We call $B(J^{\Delta x})$ and $B(\G^{\Delta })$ the spaces of bounded functions defined respectively on $J^{\Delta x}$ and on  $\G^{\Delta }$.
To compute the value function on the foot of the discrete trajectories, which, in general, are not grid nodes, we approximate these values by a piecewise linear Lagrange interpolation $\mathbb I[\hat u](z)$, where $u\in B(J^{\Delta x})$ and  $z\in J$. The basic properties of the interpolation operator are summarized in the following lemma (for the proof see for instance \cite{QSS07}).

\begin{lemma}\label{inter}
Given the \emph{piecewise linear interpolation} operator $\I:B(J^{\Delta x})\times J\rightarrow \R$ and a function $\varphi \in C(J)$,  we denote by $\hat{\varphi}$ the  collection of  values $\{\varphi(x_k)\}_{x_k\in J^{\Delta x}}$. We have the following properties:
\begin{itemize}
\item (Monotonicity) If $\psi \in C(J)$ such that $\psi(x)\leq \varphi(x)$ for every $x\in J^{\Delta x}$ then
$$ \I[\hat \psi](x)\leq \I[\hat \varphi](x), \quad \hbox{for all }x\in J.$$
\item (Polynomial base) There exists a set $\{\phi_i\}_{i=I}$ ($I$ is the set of the indexes of the elements of $J^{\Delta x}$) of lagrangian bases \cite{QSS07}, such that $\phi_i(x_k)=\delta_{i,k}$ where $\delta$ is the Kronecker symbol; and 
$$ \I[\hat \varphi](x)= \sum_{i\in I} \varphi(x_i)\phi_i(x), \quad \hbox{for all }x\in J.$$
\item (Error estimate) If $\varphi \in W^{s,\infty}(J)$ with $s=1,2$, then there exists a constant $C>0$ such that
$$ |\I[\hat \varphi](x)-\varphi(x)|\leq C \Delta x^s,\quad \hbox{for all  } x \in J.$$
\item (Error estimate, smooth case) If $\varphi \in C^2(J)$ then
\begin{equation}\label{interperr}
|\I[\hat \varphi](x)-\varphi(x)|\leq   \frac{1}{2}\sup_{\xi\in [a,b]}|\varphi_{xx}(\xi)|\;\left|\prod_{i=1}^2 (x-x_i)\right|,
\end{equation}
 for all $x \in [x_1=a,x_2=b]\in J$.
 \end{itemize}
\end{lemma}

We finally define a fully discrete numerical operator $S:B(\G^{\Delta })\times J^{\Delta x}\to \R$  as, if $x\in J_i$
\begin{equation*}\label{eq:slscheme2}
S[\hat  v](x):=\min \left\{
   \begin{array}{ll}
    \inf\limits_{\alpha,\alpha_i<\frac{|x|}{\Delta t}}\mathbb I[\hat  v](x-\alpha_i\Delta t e_i)+\Delta t L_i(x,\alpha_i),\\
  \inf\limits_{\alpha,\alpha_i\geq \frac{|x|}{\Delta t}}\inf\limits_{ s_0\in[0,\Delta t-\frac{|x|}{\alpha_i}]} \min\limits_{j,\alpha_j\leq 0}\left\{\mathbb I[\hat  v]\left(-\left(\Delta t- s_0-\frac{|x|}{\alpha_i}\right)\alpha_j e_j\right)\right.\\
 \left.\phantom{tu} +\left(\Delta t- s_0-\frac{|x|}{\alpha_i}\right)L_j(0,\alpha_j)+ s_0 L_0(\alpha_0)+\frac{|x|}{\alpha_i}L_i(x,\alpha_i)\right\},
\end{array}\right. 
\end{equation*}
and, if $x=0$
\begin{multline*}
S[\hat  v](x):=
  \inf\limits_{\alpha,j,\alpha_j\leq 0}\inf\limits_{ s_0\in[0,\Delta t]} \left\{\mathbb I[\hat  v]\left(-\left(\Delta t- s_0\right)\alpha_j e_j\right)\right. \\
  \left.+\left(\Delta t- s_0\right)L_j(0,\alpha_j)+ s_0 L_0(\alpha_0)\right\}.
\end{multline*}
We define recursively the discrete solution $w\in B(\G^{\Delta })$ as
\begin{equation}\label{eq:slscheme}
w(t_{n+1},x)=S[\hat w^n](x),\quad n=0,\dots,N_T-1,\; x\in J^{\Delta x}
\end{equation}
where $w^n:=\{w(t_n,x)\}_{x\in J^{\Delta x}}$ for $n=0,\dots,N_T-1$ and $w^0=\{u_0(x)\}_{x\in J^{\Delta x}}$.

\medskip
Next, we prove some basic properties satisfied by the scheme \eqref{eq:slscheme}, assuming that assumptions ({\bf{A1}})-({\bf{A2}}) hold.
\begin{proposition}[Monotonicity]\label{monotonicity}
The numerical scheme \eqref{eq:slscheme} is monotone, i.e. given two discrete functions $v_1, v_2\in B(J^{\Delta x})$ such that $v_1\leq v_2$ we have 
$$ S[\hat v_1](x)\leq S[\hat v_2](x), \quad \forall x\in J^{\Delta x}. $$ 
\end{proposition}
\begin{proof}
Let us fix a $x\in J^{\Delta x}_i$. 
We assume that the trajectory relative to $v_1$ passes through the junction and the one relative to $v_2$ does not. The other cases are easier and they can be treated in a similar way. Let us call $(\bar \alpha_i, \bar s_0, \bar j, \bar \alpha_{\bar j}, \bar \alpha_0)$  the optimal strategy relative to $v_1$,  and let us call $\hat \alpha_i$ the optimal control relative to $v_2$.  The optimal controls are bounded since Prop. \ref{Hproperty} holds. We have 
 \[\begin{split} S[\hat v_1](x)=\mathbb I[\hat v_1]\left(-\left(\Delta t-\bar s_0-\frac{|x|}{\bar\alpha_i}\right)\bar \alpha_{\bar j} e_{\bar j}\right)+\left(\Delta t-\bar s_0-\frac{|x|}{\bar \alpha_i}\right)L_j(0,\bar \alpha_{\bar j})\\
+\bar s_0 L_0(\bar \alpha_0)+\frac{|x|}{\alpha_i}L_i(x,\hat \alpha_i)
\leq \mathbb I[\hat v_1](x-\hat\alpha_i\Delta t e_i)+\Delta t L_i(x,\hat\alpha_i) 
= S[\hat v_2](x).
\end{split}\]
\qed
\end{proof}
\begin{proposition}\label{lipw}
   Let $w(t_n,x)$ be a solution of \eqref{eq:slscheme}. If $u_0$ is uniformly Lipschitz continuous then for $x,y\in J^{\Dx}$ there exists a $C>0$ such that 
   $$ \left|w(t_n,x)-w(t_n,y)\right|\leq C\, (\Delta t+ d(x,y)) \quad n=0,\dots,N_t$$
\end{proposition}

\begin{proof}
 In this proof, we denote by $C$ a universal constant that depends only on $L_i$ and that may change line to line and with $L_f$ the Lipschitz constant of a generic function $f$.\\
Let just assume that $x,y\in J_i\cap J^{\Delta x} $. The latter is not restrictive since  if $x\in J_j\cap J^{\Delta x}, y\in J_i \cap J^{\Delta x}$ with $j\neq i$, we come back to the case of a comparison between point belonging to the same arc writing
     \begin{equation*}\left|w(t_{n},x)- w(t_{n},y)\right|\leq \left|w(t_{n},x)- w(t_{n},0)\right| + \left|w(t_{n},0)- w(t_{n},y)\right|.
     \end{equation*}
 We call $\bar \alpha_i$ the optimal control of $S[w^{n-1}](y)$ associated to the i-arc. We consider three different cases:
 \begin{enumerate}
  \item[$1)$] $\bar \alpha_i< |y|/\Dt$ with $y\ne 0$.\newline
   In this case, we consider $\alpha_i$ such that
  $$x-\Delta t \alpha_i e_i=y-\Delta t\bar \a_i e_i.$$
  This means in particular that
  \begin{equation}\label{eq:100}
  |\a_i-\bar \a_i|=\frac {|x-y|}{\Delta t}
  \end{equation}
  Using  the suboptimal control $\a_i$ for $S[\hat w^{n-1}](x)$ yields
  \begin{multline*}
   w(t_{n},x)-w(t_{n},y)   
   \leq  \I [\hat w^{n-1}]\left(-\left(x-\alpha_i \Delta t e_i\right)\right) + \Delta tL_i(x,\alpha_i)\\
   - \I [\hat w^{n-1}]\left(-\left(y-\bar \alpha_i \Delta t e_i\right)\right) -\Delta t L_i(y,\bar \alpha_i)
   \le \Delta t L_{L_i}|\a_i-\bar \a_i|\le L_{L_i}|x-y|
   \end{multline*}
  
    \item[$2)$] $0<\frac {|y|}{\Delta t}\le \bar \a_i$. This means in particular that the discrete trajectory starting from $y$ passes through the junction. We denote by $(\bar \a_i,\bar s_0,\bar \a_0,\bar j,\bar \a_{\bar j})$ the optimal control associated with  $S[\hat w^{n-1}](y)$.  We distinguish two sub cases:
    \begin{itemize}  
    \item[2.i)] $x=0$. 
   In this case, we choose the suboptimal control $(\bar s_0+\frac {|y|}{\bar \alpha_i},\bar \a_0,\bar j,\bar \a_{\bar j})$ (if $\bar \a_0\ne 0$, we replace it by $0$ in order to stay in the origin) and get    
   \begin{align}\label{eq:101}
     &w(t_{n},x)-w(t_{n},y)   \nonumber\\
   \le &\I [\hat w^{n-1}]\left(-\left(\Dt -\bar s_0-\frac {|y|}{\bar \alpha_i}\right)\bar \alpha_{\bar j} e_{\bar j}\right) +\left(\Dt -\bar s_0-\frac {|y|}{\bar \alpha_i}\right)L_{\bar j}(0,\bar \alpha_{\bar j})\nonumber\\
  & +\left(\bar s_0+\frac {|y|}{\bar \alpha_i}\right) L_0(\bar \alpha_0)
   -\I [\hat w^{n-1}]\left(-\left(\Delta t-\bar s_0-\frac {|y|}{\bar \a_i}\right)\bar \a_{\bar j} e_{\bar j}\right)\nonumber\\
   &-\left(\Delta t-\bar s_0-\frac {|y|}{\bar \alpha_i}\right)L_{\bar j}(0,\bar \a_{\bar j})\
-\bar s_0 L_0(\bar \alpha_0)-\frac{|y|}{\bar \a_i}L_i(y,\bar \a_i)\nonumber\\
\le & \frac {|y|}{\bar \alpha_i}\left(L_0(\bar \a_0)-L_i(y,\bar \a_i)\right).
   \end{align}
   If $\bar \a_i\ge 1$, using that $L_i$ is Lipschitz continuous, we get that there exists a constant $C$ (depending only on $L_i(y,0)$ and the Lipschitz constant of $L_i$) such that
   $$\frac{|L_i(y,\bar \a_i)|}{|\bar \a_i|}\le C.$$
   Injecting the estimate above in \eqref{eq:101} and using that $L_0(0)$ is bounded, we get
   $$w(t_{n},x)-w(t_{n},y)\le C|y|=C d(x,y).$$
   If $\bar \a_i\le 1$, since $L_0(0)$ and $L_i(\bar \a_i)$ are bounded, we get that there exists a constant $C$ such that
   $$w(t_{n},x)-w(t_{n},y)\le C\frac {|y|}{\bar \alpha_i}\le C \Delta t.$$
   We finally get that in all the cases,
   $$w(t_{n},x)-w(t_{n},y)\le  C \left(\Delta t+d(x,y)\right).$$
\item[2.ii)] $|x|>0$.  
   In this case, we choose $\a_i$ such that $\frac{|x|}{\a_i}=\frac{|y|}{\bar \a_i}$. This implies in particular that
   $$x-\frac {|x|}{\a_i} \a_i e_i=y-\frac {|x|}{\a_i} \bar \a_ie_i$$ 
   and so
   $$\frac {|x|}{\a_i} |\a_i-\bar \a_i|=|x-y|=d(x,y).$$
  Using the suboptimal control $(\a_i,\bar s_0,\bar \a_0,\bar j,\bar \a_{\bar j})$ for $S[w^{n-1}](x)$, we get
\begin{multline*}
     w(t_{n},x)-w(t_{n},y)   \nonumber\\
   \le \I [\hat w^{n-1}]\left(-\left(\Dt -\bar s_0-\frac {|x|}{\alpha_i}\right)\bar \alpha_{\bar j} e_{\bar j}\right) +\left(\Dt -\bar s_0-\frac {|x|}{\alpha_i}\right)L_{\bar j}(0,\bar \alpha_{\bar j})\\ 
   +\bar s_0 L_0(\bar \alpha_0)  +\frac {|x|}{\a_i}L_i(x,\a_i)  -\I [\hat w^{n-1}]\left(-\left(\Delta t-\bar s_0-\frac {|y|}{\bar \a_i}\right)\bar \a_{\bar j} e_{\bar j}\right)\nonumber\\
-\left(\Delta t-\bar s_0-\frac {|y|}{\bar \alpha_i}\right)L_{\bar j}(0,\bar \a_{\bar j})-\bar s_0 L_0(\bar \alpha_0)-\frac{|y|}{\bar \a_i}L_i(y,\bar \a_i)\nonumber\\
\le  L_{L_i}\frac {|x|}{\alpha_i}|\a_i-\bar \a_i|
\le  L_{L_i} d(x,y)
   \end{multline*}
   \end{itemize}
   
   \item[$3)$] $y=0$.  We denote by $(\bar s_0,\bar \a_0,\bar j,\bar \a_{\bar j})$ the optimal control associated to the operator $S[w^{n-1}](y)$.
   We distinguish two sub-cases again:
   \begin{itemize}
   \item[3.i)]: {$\bar s_0=\Delta t$}. 
   We choose $\a_i\ge \max(1,\frac {|x|}{\Delta t})$ and the suboptimal control $(\a_i,\bar s_0-\frac {|x|}{\a_i},\bar \a_0)$ for $S[w^{n-1}](x)$. We then get
   \begin{multline*}
     w(t_{n},x)-w(t_{n},y)   \nonumber\\
   \le \I [\hat w^{n-1}]\left(0\right) +\left(\bar s_0-\frac{|x|}{\a_i}\right) L_0(\bar \alpha_0)+\frac {|x|}{\a_i}L_i(x,\a_i)     -\I [\hat w^{n-1}]\left(0\right)\\-\bar s_0 L_0(\bar \alpha_0)
\le  \frac {|x|}{\a_i}\left(L_i(x,\a_i)-L_0(\bar \a_0)  \right)
\le  L_{L_i} d(x,y)
   \end{multline*}
    Using that $L_i$ is Lipschitz continuous, we get that there exists a constant $C$ (depending only on $L_i(0), L_0(0)$ and on the Lipschitz constant of $L_i$) such that
   $$\frac{|L_i(x,\bar \a_i)|+|L_0(\bar \a_0)|}{|\bar \a_i|}\le C.$$
   This implies that 
   $$w(t_{n+1},x)-w(t_{n+1},y)\le C|x|=Cd(x,y).$$
 
   \item[3.ii)]: {$\bar s_0<\Delta t$}. 
We choose $\a_i\ge \max(1,|\bar \a_{\bar j}|)$ such that
\begin{equation}\label{eq:102}
\frac {|x|}{\a_i}\le \frac{\Delta t-\bar s_0} 2\quad{\rm and}\quad \frac{\Delta t-\bar s_0}{\Delta t-\bar s_0-\frac{|x|}{\a_i}}|\bar \a_{\bar j}|\le \a_i.
\end{equation}   
We also set 
$$\a_{\bar j}=\frac{\Delta t-\bar s_0}{\Delta t-\bar s_0-\frac{|x|}{\a_i}}\bar \a_{\bar j},$$
which satisfies in particular $\a_i\ge |\a_{\bar j}|.$ Taking the suboptimal control $(\a_i,\bar s_0,\bar \a_0,\bar j,\a_{\bar j})$ for $S[w^{n-1}](x)$, we get
  \begin{multline}
     w(t_{n},x)-w(t_{n},y)  
   \le \I [\hat w^{n-1}]\left(-\left(\Dt -\bar s_0-\frac {|x|}{\alpha_i}\right) \alpha_{\bar j} e_{\bar j}\right)\\
    +\left(\Dt -\bar s_0-\frac {|x|}{\alpha_i}\right)L_{\bar j}( 0,\alpha_{\bar j})
   +\bar s_0 L_0(\bar \alpha_0) +\frac {|x|}{\a_i}L_i(x,\a_i) \\ 
   -\I [\hat w^{n-1}]\left(-\left(\Delta t-\bar s_0\right)\bar \a_{\bar j} e_{\bar j}\right)-\left(\Delta t-\bar s_0\right)L_{\bar j}(0,\bar \a_{\bar j})-\bar s_0 L_0(\bar \alpha_0)\\
\le  \frac {|x|}{\alpha_i}\left(L_i(x,\a_i)-L_{\bar j}(0,\a_{\bar j})  \right) + (\Delta t-\bar s_0)\left(L_{\bar j}(0,\a_{\bar j})-L_{\bar j}(0,\bar \a_{\bar j})\right)\\
\le\frac {|x|}{\alpha_i}\left(L_i(x,\a_i)-L_{\bar j}(0,\a_{\bar j})  \right) + (\Delta t-\bar s_0)L_{L_{\bar j}}|\a_{\bar j}- \bar \a_{\bar j}|.\label{eq:103}
   \end{multline}
 Using that $\alpha_i\ge 1$, we get that $\frac {|L_i(x,\a_i)}{\a_i}\le C$. In the same way (using that $\a_i\ge \a_{\bar j}$)
 $$\frac{|L_{\bar j}(0,\a_{\bar j})|}{\a_i}\le \frac 1{\a_i}\left(L_{\bar j}(0,0)+L_{L_{\bar j}} |\a_{\bar j}|\right)\le L_{\bar j}(0,0)+ L_{L_{\bar j}} \frac{|\a_{\bar j}|}{\a_i}\le C.$$
   Finally, using the definition of $\a_{\bar j}$, we get
   $$(\Delta t-\bar s_0) |\a_{\bar j}- \bar \a_{\bar j}|=|x|\frac {|\a_{\bar j}|}{\a_i}\le |x|.$$
   Injecting these estimates in \eqref{eq:103}, we arrive to
   $$w(t_{n},x)-w(t_{n},y) \le C|x|=Cd(x,y).$$
   \end{itemize}
  \end{enumerate}
\qed\end{proof}
%
\begin{proposition}[Stability]\label{stability}
Let $w(t_n,x)$ be a solution of \eqref{eq:slscheme}, then there is a positive constant $K$ such that for any $(t_n,x)\in \mathcal{G}^{\Delta }$
 $$|w(t_n,x)-u_0(x)|\leq  K t_n .$$

\end{proposition}
\begin{proof}
Let $K\in \R$ be such that  
$$K\geq \max \left\{ \sup_{x\in J^{\Delta x}}\frac{[S[\hat u^0](x)-u_0(x)]^+}{\Dt},\sup_{x\in J^{\Delta x}}\frac{[u_0(x)-S[\hat u^0](x)]^+}{\Dt}\right\},$$
where $u^0:=\{u_0(x)\}_{x\in J^{\Delta}}$.
The discrete function $\overline{u}(x,t_n):=u_0(x)+K t_n$ is a super discrete solution, i.e. $\overline{u}(x,t_{n+1})\geq S[\hat {\overline{u}}(\cdot,t_n)](x)$ for all $(x,t_n)\in \mathcal{G}^{\Delta }$.
In fact, since for all $(x,t_n)\in \mathcal{G}^{\Delta }$
 $$K\Dt+K t_n\geq \sup_{x\in J^{\Delta x}} [S[\hat u^0](x)-u_0(x)]^++K t_n\geq S[\hat u^0](x)-u_0(x)+K t_n,$$
  we have
 $$
 \overline{u}(x,t_{n+1})=u_0(x)+K t_{n+1}\geq u_0(x)+S[\hat u^0](x)-u^0(x)+K t_n= S[\hat {\overline{u}}(\cdot,t_n)](x).
 $$
The discrete function $\underline{u}(x,t_n):=u_0(x)-K t_n$ is a sub discrete solution, i.e. 
$$\underline{u}(x,t_{n+1})\leq S[\hat {\underline{u}}(\cdot,t_n)](x), \hbox{for all }(x,t_n)\in \mathcal{G}^{\Delta }.$$
In fact, since for all $(x,t_n)\in \mathcal{G}^{\Delta }$
$$K\Dt+K t_n\geq \sup_{x\in J^{\Delta x}} [u_0(x)-S[\hat u^0](x)]^++K t_n\geq u_0(x)-S[\hat u^0](x)+K t_n,$$ we have
 $$
 \underline{u}(x,t_{n+1})=u_0(x)-K t_{n+1}\geq u_0(x)-(u_0(x)-S[\hat u^0](x)+K t_n)= S[\hat {\underline{u}}(\cdot,t_n)](x).
 $$ 
By monotonicity, we have that $\underline{u}(x,t_n)\leq w(x,t_n)\leq \overline{u}(x,t_n)$ for any $(x,t_n)\in \mathcal{G}^{\Delta }$ and this implies the conclusion.

\qed\end{proof}

\begin{remark}[Bounded control] By Prop. \ref{stability},  $w$ solution of \eqref{eq:slscheme} is bounded and then the discrete problem \eqref{eq:slscheme} is well posed. We observe also that the same argument of Proposition \ref{Hproperty} (based on \textbf{(A2)}) can be used to prove that the control $\alpha$ in \eqref{eq:slscheme} is bounded. We call \begin{equation}\label{CFL}
\mu=\sup\limits_{(x,t)\in J\times(0,T]}\max\limits_{i=1,...,n}|\alpha_i^*|,
\end{equation}
the maximal absolute value of the optimal control.
\end{remark}

\begin{proposition}\label{consisrates}
Given $\Delta t>0$ and $\Delta x>0$, let us assume
\begin{equation}\label{CFLg}
{\mu}\frac{ \Delta t}{\Delta x}\le 1
\end{equation}
(with $\mu$ as in \eqref{CFL}), then
for any   $\varphi \in C^2(J)$ 
the following consistency error estimates hold for the scheme \eqref{eq:slscheme}:
\be\label{cons31}
i)\hbox{ if } x\in J_i^*, \quad \left|\frac{\varphi(x) - S[\hat \varphi](x)}{\Dt}-H_i(\varphi_x(x))\right|\leq K\|\varphi_{xx}\|_\infty \left(\frac{\Dx^2}{\Dt}+\Dt\right),\ee
\be\label{cons3}
ii)\hbox{ if } x=0,\quad\left|\frac{\varphi(x) - S[\hat \varphi](x)}{\Dt}- F_{A}(\varphi_{x}(x))\right|\leq K\|\varphi_{xx}\|_\infty\left( \frac{\Dx^2}{\Dt}+\Dt\right),\ee 
where $K$ is a positive constant.
\end{proposition}

\begin{proof}

\noindent $i)$ Let $x\in J_i$. We remark that the condition \eqref{CFLg} implies in particular that the scheme reads
\begin{align*}S[\hat \varphi](x)=&  \inf_{\alpha_i< \frac{|x|}{\Dt} }(\I[\hat \varphi](x-\Dt \alpha_i e_i )+\Dt L_i(x,\alpha_i) )\\
=&  \inf_{\alpha_i\in \R }(\I[\hat \varphi](x-\Dt \alpha_i e_i )+\Dt L_i(x,\alpha_i) ).
\end{align*}
By \eqref{interperr} and by Taylor expansion we have 
$$\I[\hat \varphi](x-\Dt \alpha_i e_i )=\varphi(x)-\Dt \alpha_i \partial_i \varphi(x)+\mathcal O(\Dt^2+\Delta x^2),$$ then
\begin{align*}
\frac{\varphi(x) - S[\hat \varphi](x)}{\Dt}=&-\inf_{\alpha_i\in \R}\left(-\alpha_i \partial_i \varphi(x)+ L_i(x,\alpha_i) \right)+\mathcal O \left(\Dt+\frac{\Dx^2}{\Dt}\right)\\
=&\sup_{\alpha_i\in \R }\left(\alpha_i \partial_i \varphi(x)- L_i(x,\alpha_i) \right)+\mathcal O\left(\Dt+\frac{\Dx^2}{\Dt}\right)\\
=&H_i(  \varphi_x(x))+\mathcal O\left(\frac{\Dx^2}{\Dt}+\Dt\right).
\end{align*}
\medskip

\noindent $ii)$ Let  $x=0$.
In this case
$$S[\hat \varphi](0)=  \inf_{s_0 \in [0,\Dt] }\min_{j, \alpha_j\leq 0 } (\I[\hat \varphi](-(\Dt- s_0) \alpha_j e_j )+(\Dt- s_0) L_j(0,\alpha_j) + s_0  L_0(\alpha_0)).$$
Let us define $K_{\Dt}:=\frac{s_0}{\Dt}$,
since $s_0\in [0,\Dt]$ we have $K_{\Dt}\in[0,1]$. 
Again by Taylor expansion, by Prop. \ref{Hproperty}   and by \eqref{interperr}, we have 
\begin{align*}
&\frac{\varphi(0) - S[\hat \varphi](0)}{\Dt} +\mathcal O \left(\Dt+\frac{\Dx^2}{\Dt}\right)\\
=&-  \inf_{{{K_{\Dt} \in [0,1]}} }\min_{j, \alpha_j\leq 0 }\left(-(1-K_{\Dt})\alpha_j \partial_j \varphi(0)+(1-K_{\Dt}) L_j(0,\alpha_j)+K_{\Dt}L_0(\alpha_0) \right) \\
=&-  \inf_{K_{\Dt} \in [0,1] }\left[(1-K_{\Dt})\underset{  j ,\alpha_j\leq 0}\min\left(-\alpha_j \partial_j \varphi(0)+ L_j(0,\alpha_j)\right)+K_{\Dt}\min_{\alpha_0}\left( L_0(\alpha_0) \right)\right]\\
=&\sup_{K_{\Dt} \in [0,1] }\left[(1-K_{\Dt})\underset{ j ,\alpha_j\leq 0}\max\left(\alpha_j \partial_j \varphi(0)- L_j(0,\alpha_j)\right)+K_{\Dt}\max_{\alpha_0}\left(- L_0(\alpha_0) \right)\right] \\
=&\sup_{K_{\Dt} \in [0,1] }\left\{(1-K_{\Dt})\,\underset{  j }\max{\, H_j^-(\partial_j \varphi(0))}+K_{\Dt}A \right\} \\
=&\max\left(\underset{ j }\max\, {H_j^-(\partial_j \varphi(0))},A\right).
\end{align*}
This ends the proof of the proposition.

\qed\end{proof}

\begin{remark}[Counterexample for consistency errors]
The case that we study behaves differently from classic SL schemes, where consistency error estimate is not limited by a CFL-like condition. This difference is due to the presence of discontinuities on the Hamiltonians in correspondence with the junction point.   We provide a counterexample clarifying the scenario.

Let us consider a simple junction  $J:=J_1\cup J_2=(-\infty,0]\cup [0,+\infty)$, provided with the Hamiltonians
$$ H_1(p)=\frac{|p|^2}{2}-1=\max_{\alpha_1}\left(\alpha_1\cdot p-\frac{\alpha_1^2}{2}-1\right),\; H_2(p)=\max_{\alpha_2}\left(\alpha_2\cdot p-\frac{\alpha^2_2}{2}-2\right),$$
and the flux limiter $A=-1=\max\min_p H_i(p)$ (small enough to not be considered).
We check the consistency at $x=\Delta x$ for the smooth function $\varphi=1-x$.
We can check that the scheme \eqref{eq:slscheme} reads, if $\alpha_2\leq\Delta x/\Delta t$
\begin{equation}\label{sc1}
S[\varphi](\Delta x)=\inf_{\alpha_2}\left(1+(\Delta x-\alpha_2 \Delta t)+\Dt\left(\frac{\alpha_2^2}{2}+2\right)\right)=1+\Dx+\frac{3}{2}\Dt
\end{equation}
(the minimum is reached for $\alpha_2=1$) and if $\alpha_2>\Delta x/\Delta t$
\begin{multline}\label{sc2}
S[\varphi](\Delta x)=\\
\inf_{\alpha_2}\inf_{\alpha_1}\left[1+\left(\Dt-\frac{\Dx}{\alpha_2}\right)\alpha_1+\left(\Dt-\frac{\Dx}{\alpha_2}\right)\left(\frac{\alpha_1^2}{2}+1\right)+\frac{\Dx}{\alpha_2}\left(\frac{\alpha_2^2}{2}+2\right)\right]\\
=1+\sqrt{3}\Delta x+\frac{\Delta t}{2}, 
\end{multline}
where the minimum corresponds to $\alpha_1=-1$ and $\alpha_2=\sqrt{3}$. The minimum between the two options depends on the rate $\Dt/\Dx$: simply comparing the two output we notice that for $\frac{\Dt}{\Dx}\leq \sqrt{3}-1$ the scheme assumes the form \eqref{sc1}, while if $\frac{\Dt}{\Dx}> \sqrt{3}-1$, \eqref{sc2}. We notice that \eqref{sc1} is consistent with the equation  since
\begin{multline*}\frac{\varphi(\Dx)-S[\varphi](\Delta x)}{\Dt}-H_2(\varphi_x(\Delta x))=\\
\frac{1+\Delta x-(1+\Dx+3/2\Dt)}{\Dt}-\left(\frac{|\varphi_x(\Delta x)|^2}{2}-2\right)=0
\end{multline*}
while instead for \eqref{sc2}
\begin{multline*}\frac{\varphi(\Dx)-S[\varphi](\Delta x)}{\Dt}-H_2(\varphi_x(\Delta x))=\\
\frac{1+\Delta x-(1+\sqrt{3}\Dx+1/2\Dt)}{\Dt}-\left(\frac{|\varphi_x(\Delta x)|^2}{2}-2\right)=1+\sqrt{3}\left(\frac{\Dx}{\Dt}\right).
\end{multline*}
The latter means that if $\frac{\Dt}{\Dx}> \sqrt{3}-1$, no consistency error can be found for the scheme \eqref{eq:slscheme}.
\end{remark}

It is worth to underline that consistency (without consistency error estimate) holds \emph{without assuming} \eqref{CFLg}, and consequently the scheme is convergent without any CFL condition, as we show at the end of this section.   

\begin{definition}\label{def.cons}

Let $x\in J$ and $(\Delta x_m ,\Dt_m)\to0$ as $m\to \infty$. Let $y_m\in J^{\Delta x_m}$ be  a sequence of grid points  such that $y_m\to x$ as $m\to \infty$.
 The scheme $S_{\Dt }$ is  said to be {\bf{consistent}} with \eqref{eq:hjnet} if
 the following properties hold:
 \begin{itemize}
 \item[$i)$] If $x\in J_i$, for all test function $\varphi \in C^2(J)$, we have
\be\label{cons}
\frac{\varphi(y_m) - S[\hat \varphi](y_m)}{\Dt}\to H_i(\varphi_x(x))\quad{\textrm{as}}\;m\to \infty,\ee 
\item[$ii)$] If $x=0$, for all test function $\varphi \in C^2(J)$ such that $\partial_i \varphi(0)=p_i^A$ for $i=1,\dots,N$,
where $p_i^A\in \R$ are such that $H_i(p_i^A)=H^+_i(p_i^A)=A$ and $H_i^+(p):=\sup _{\alpha_i\geq 0}(\alpha_i \cdot p-L_i(\alpha))$, we have
\be\label{cons2}
\frac{\varphi(y_m) - S[\hat \varphi](y_m)}{\Dt_m}\to F_{A}(\varphi_{x}(x)) \quad{\textrm{as}}\;m\to \infty.\ee 

\end{itemize}
\end{definition}

\begin{proposition}\label{consis}
Assume that $\Delta x^2/\Dt\to 0$
Then, the scheme \eqref{eq:slscheme} is consistent according to Definition \ref{def.cons}.
\end{proposition}
\begin{proof}
Let us consider a sequence $y_m$ such that $y_m\to x$ as $\Delta_m\rightarrow (0,0)$.
For notational convenience we drop the index $m$ of the sequence of grid points. 
In case the limit point $x$ is not on the junction since $x$ is fixed  for every sequence $(\Dx_m,\Dt_m)\rightarrow (0,0)$, $y$ definitely verifies $|y|>\mu\Delta t_m$ \emph{independently} from the rate $\Dt_m/\Dx_m$.
Then, the consistency follows as  \textbf{Case 1}  in the proof of Prop. \ref{consisrates} (without the condition $\Delta t/\Delta x\leq 1/\mu$).\medskip

The situation is more complex when the limit point $x$ is $0$.
If $y\equiv0$, this case  is equivalent to  \textbf{Case 2}  in the proof of Prop. \ref{consisrates}. If $y$ is such that $y\to 0 $ and $y \neq 0$,  up to a subsequence, we can assume that $y\in J_i$, for some $i$ independent of $m$. In that case, the optimal trajectory can cross the junction in one time step.
Let $\varphi \in C^2(J)$ such that $\partial_i \varphi(0)=p_i^A$ for $i=1,\dots,N$ and let us define the two quantities:
\begin{align*}
\mathcal I_1:=&\inf_{\alpha_i< \frac{|y|}{\Dt} }(\I[\hat \varphi](y-\Dt \alpha_i e_i )+\Dt L_i(y,\alpha_i) ),\\
\mathcal I_2:= &\inf\limits_{\alpha,\alpha_i\geq \frac{|y|}{\Delta t}}\inf\limits_{ s_0\in[0,\Delta t-\frac{|y|}{\alpha_i}]} \min\limits_{j,\alpha_j\leq 0}\left\{ \I[\hat \varphi]\left(-\left(\Delta t- s_0-\frac{|y|}{\alpha_i}\right)\alpha_j e_j\right)\right.\\
 &\left.\hspace{1cm}+\left(\Delta t- s_0-\frac{|y|}{\alpha_i}\right)L_j(\alpha_j)+ s_0 L_0(\alpha_0)+\frac{|y|}{\alpha_i}L_i(\alpha_i)\right\}. 
 \end{align*}
 We remark that $S[\varphi](y)=\min(\mathcal I_1,\mathcal I_2).$
We begin with the term $\mathcal I_1$. Using \eqref{interperr} and a Taylor expansion, we get
\begin{align}\label{eq:est-I1-1}
\mathcal I_1=&\inf_{\alpha_i\le \frac {|y|}{\Dt}}\left\{\varphi(y)-\a_i\Dt \partial _i\varphi(y)+\Dt L_i(y,\a_i)\right\} +\mathcal O(\Dx^2+\Dt^2)\nonumber\\
=&\varphi(y)-\Dt \sup _{\alpha_i\le \frac {|y|}{\Dt}}\left\{\a_i \partial _i\varphi(y)- L_i(y,\a_i)\right\} +\mathcal O(\Dx^2+\Dt^2)
\end{align}
Using that 
\begin{align*}
\sup _{\alpha_i\le \frac {|y|}{\Dt}}\left\{\a_i \partial _i\varphi(y)- L_i(y,\a_i)\right\} \le &\sup _{\alpha_i\in \R}\left\{\a_i \partial _i\varphi(y)- L_i(y,\a_i)\right\} \\
=& H_i(y,\partial_i\varphi(y))=A+o(1),
\end{align*}
 we deduce that
\be\label{eq:est-I1-2}
\mathcal I_1\ge \varphi(y) - \Dt A +\Dt\, o(1) +\mathcal O(\Dx^2+\Dt^2).
\ee
For the term $\mathcal I_2$, with \eqref{interperr}, adding into the argument of $\varphi$ the term $y-\frac{|y|}{\alpha_i}\alpha_i e_i=0$ and using the Taylor expansion twice we obtain
 \begin{multline*}\I[\varphi]\left(-\left(\Delta t- s_0-\frac{|y|}{\alpha_i}\right)\alpha_j e_j\right)=\\
 \varphi(y)-\frac{|y|}{\alpha_i}\alpha_i \partial_i\varphi(y)-\left(\Delta t- s_0-\frac{|y|}{\alpha_i}\right)\alpha_j \partial_j\varphi(0)+\mathcal O(\Delta t^2+\Dx^2).
\end{multline*}
This implies
\begin{align*}
&\mathcal I_2+\mathcal O(\Dt^2+\Dx^2)\\
=&\inf_{\alpha_i\geq \frac{|y|}{\Delta t}} \inf_{s_0 \in [0,\Dt-\frac{|y|}{\a_i}] }\bigg\{\min_{j}\min_{\alpha_j\leq 0 }\left\{-\left(\Delta t- s_0-\frac{|y|}{\alpha_i}\right)(\alpha_j \partial_j \varphi(0) - L_j(0,\alpha_j))\right\}\\
&\quad \hspace{2cm}+\varphi(y)-\frac {|y|}{\a_i}(\alpha_i \partial_i \varphi(y) - L_i(y,\alpha_i))+s_0L_0(\alpha_0) \bigg\}\\
&=\varphi(y)+\inf_{\alpha_i\geq \frac{|y|}{\Delta t}} \inf_{s_0 \in [0,\Dt-\frac{|y|}{\a_i}] }\bigg\{-\frac {|y|}{\a_i}(\alpha_i \partial_i \varphi(y) - L_i(y,\alpha_i))+s_0L_0(\alpha_0)\\
&\quad\hspace{2cm}-\left(\Delta t- s_0-\frac{|y|}{\alpha_i}\right)\max_{j}\max_{\alpha_j\leq 0 }\left\{(\alpha_j \partial_j \varphi(0) - L_j(0,\alpha_j))\right\}\bigg\}\\
&=\varphi(y)+\inf_{\alpha_i\geq \frac{|y|}{\Delta t}} \inf_{s_0 \in [0,\Dt-\frac{|y|}{\a_i}] }\bigg\{-\frac {|y|}{\a_i}(\alpha_i \partial_i \varphi(y) - L_i(y,\alpha_i))+s_0L_0(\alpha_0)\\
&\quad\hspace{2cm}-\left(\Delta t- s_0-\frac{|y|}{\alpha_i}\right) \max_j H_j^-(0,\partial _j\varphi (0))\bigg\}
\end{align*}
Using $\max_j H_j^-(0,\partial _j\varphi (0))=\max_j \min_p H_j(p)=\bar H^0$, and $L_0(\alpha_0)=-A$, we deduce that (we use $\bar H^0\le A$)
\begin{align}\label{eq:est-I2-1}
&\mathcal I_2+\mathcal O(\Dt^2+\Dx^2)\nonumber\\
&=\varphi(y)+\inf_{\alpha_i\geq \frac{|y|}{\Delta t}}\bigg\{-\frac {|y|}{\a_i}(\alpha_i \partial_i \varphi(y) - L_i(y,\alpha_i)) \nonumber\\
&\hspace{3cm}+\inf_{s_0 \in [0,\Dt-\frac{|y|}{\a_i}] }\left\{s_0(\bar H^0-A) -\left(\Dt -\frac{|y|}{\a_i}\right)\bar H^0\right\}\bigg\}\nonumber\\
&=\varphi(y)+\inf_{\alpha_i\geq \frac{|y|}{\Delta t}}\bigg\{-\frac {|y|}{\a_i}(\alpha_i \partial_i \varphi(y) - L_i(y,\alpha_i)) -\left(\Dt -\frac{|y|}{\a_i}\right)A\bigg\}\\\nonumber
&=-\Dt\, A+\varphi(y)-\sup_{\alpha_i\geq \frac{|y|}{\Delta t}}\bigg\{\frac {|y|}{\a_i}(\alpha_i \partial_i \varphi(y) - L_i(y,\alpha_i)-A) \bigg\}\nonumber
\end{align}
Using $\frac {|y|}{\a_i}\le \Dt$ in the last sup, and $\alpha_i \partial_i \varphi(y) - L_i(y,\alpha_i)-A\le o(1)$, we get
\be\label{eq:est-I2-2}
\mathcal I_2+\mathcal O(\Dt^2+\Dx^2)\ge -\Dt \, A+\varphi(y) +\Dt o(1).
\ee
Using \eqref{eq:est-I1-2} and \eqref{eq:est-I2-2}, we finally get 
\be\label{eq:estI}
S[\varphi](y)=\min(\mathcal I_1,\mathcal I_2)\ge -\Dt A+\varphi(y) +\Dt o(1)+\mathcal O(\Dt^2+\Dx^2).
\ee
We now want to show that this inequality is in fact an equality. We denote by $\bar \a_i$ the solution of
$$\sup_{\a_i\in \R}\{\alpha_i \partial_i \varphi(y) - L_i(y,\alpha_i)\},$$
and we distinguish two cases.
Firstly, we consider the case $\bar \a_i\le \frac {|y|}{\Dt}$. This implies in particular that
\begin{align*}\sup_{\a_i\le \frac {|y|}{\Dt}}\{\alpha_i \partial_i \varphi(y) - L_i(y,\alpha_i)\}=&\sup_{\a_i\in \R}\{\alpha_i \partial_i \varphi(y) - L_i(y,\alpha_i)\}\\
=&H_i(y,\partial_i \varphi (y))=A+o(1).
\end{align*}
Using \eqref{eq:est-I1-1}, we deduce that
$$\mathcal I_1= -\Dt A+\varphi(y) +\Dt o(1)+\mathcal O(\Dt^2+\Dx^2)$$
and so \eqref{eq:estI} is an equality.

We now consider the case $\bar \a_i\ge \frac {|y|}{\Dt}$. We define
$$\bar {\mathcal I}_2:=\sup_{\alpha_i\geq \frac{|y|}{\Delta t}}\bigg\{\frac {|y|}{\a_i}(\alpha_i \partial_i \varphi(y) - L_i(y,\alpha_i)-A) \bigg\}.$$
Using that $0\le\frac{|y|}{\a_i}\le\Dt$ and that $\alpha_i \partial_i \varphi(y) - L_i(y,\alpha_i)-A\le o(1)$, we get
\begin{align*}
\Dt\,  o(1)\ge \bar{\mathcal I}_2
\ge &\Dt \left\{\sup_{\alpha_i\geq \frac{|y|}{\Delta t}}\left\{\frac {|y|}{\a_i}(\alpha_i \partial_i \varphi(y) - L_i(y,\alpha_i)\right\}-A\right\}\\
=&\Dt (H_i(y,\partial _i\varphi(y))-A)=\Dt \,o(1).
\end{align*}
This implies again that  \eqref{eq:estI} is an equality and ends the proof.

\qed\end{proof}

\begin{theorem}[Convergence]
Assume that $\frac{\Dx^2}{\Dt}\to 0$ and let $T>0$ and $u_0$ be a Lipschitz continuous function on $J$. Then the numerical solution of \eqref{eq:slscheme} $w(t,x)$ converges locally uniformly as $\Delta\to (0,0)$  to the unique (weak) viscosity solution $u(t,x)$ of \eqref{eq:hjnet} on any compact set $\mathcal K$ contained in the domain $(0,T)\times J$, i.e.
$$ \limsup_{\Delta x,\Delta t\rightarrow 0} \sup_{(t,x)\in \mathcal K \cap \G^{\Delta }}|w(t,x)-u(t,x)|=0.$$
\end{theorem}

\begin{proof}
Since the scheme is consistent (for a subsequence verifying $\Dx^2/\Dt \to 0$), monotone and stable. We can follow
\cite{BS91,costeseque2015convergent,imbert2013flux} and obtain the result. Note that the choice of the test functions in the definition of the consistency at the junction is consistent with Theorem \ref{th:1} ii)
\qed\end{proof}

\section{Convergence estimates}\label{sect:bounds}
In this section, we introduce the main result of the paper. We formulate the result under two possible cases: firstly for special Hamiltonians, and secondly in a more generic scenario. 

\subsection{Space independent Hamiltonians.} 

We suppose 
\begin{itemize}
\item[(\bf{A3})]
 the lagrangians $L_i(x,\alpha_i)\equiv L_i(y,\alpha_i)$ for every choice of $x,y\in J_i$. 
 \end{itemize}
 We observe that as consequence of \textbf{(A3)} the optimal control $\bar \alpha_i$ is constant along the arc.

\begin{theorem}\label{teo:bounds1}
Let \textbf{(A1)}-\textbf{(A3)} verified. 
Considered $u$ a viscosity solution of \eqref{eq:hjnet}, and $w$ a solution of the scheme \eqref{eq:slscheme}.  Then, there exists a positive constant $C$ depending only on the Lipschitz constant of $u$ such that 
\begin{equation}
 \sup_{(t,x )\in  \mathcal G^{\Delta }}|u(t,x)-w(t,x)| \leq CT \Delta x. 
\end{equation}
\end{theorem}

\begin{proof}
The proof is made by induction assuming that for $n\ge 1$
\begin{equation}\label{eq:rec}
|w^{n-1}(x)-u(t_{n-1},x)|\le (n-1)C\Delta x\quad \forall x\in J^{\Delta x}.
\end{equation}
Note that it is clearly satisfy for $n=1$. We then want to show that 
$$|w^{n}(x)-u(t_n,x)|\le nC\Delta x\quad \forall x\in J^{\Delta x}.$$
From Proposition \ref{hjdpp}, we know that
\begin{multline}\label{eq:dppu}
u(t_n,x):=\\
\inf_{y\in J}\inf_{(X(.),\alpha(.))\in \Gamma_{t_{n-1},y}^{t_n,x}}\left\{u(t_{n-1}, y)+\int\limits_{t_{n-1}}^{t_n} L(X(\tau),\alpha(\tau))d\tau\right\}. 
\end{multline}

We call $\bar\alpha=(\bar\alpha_0,\bar\alpha_1,...,\bar\alpha_N)$ and $\bar s_0$ the optimal argument of $S[w^{n-1}](x)$ and we treat only the case where $x\in J^i\backslash\{0\}$ and  $|x|/\bar\alpha_i<\Dt$ (this corresponds to the more difficult case in which the optimal trajectory cross the junction). We also denote by $\bar X(t)$ (with $t\in[t_{n-1},t_{n}]$) the trajectory obtained applying the control $\bar\alpha$. Clearly such trajectory belongs to $\Gamma_{t_{n-1},\bar X(t_{n-1})}^{t_n,x}$ with $\bar X(t_{n-1})=\left(\Delta t-s_0-\frac{|x|}{\bar \alpha_i}\right)e_j$ and  
 \be \left\{
 \ba{ll}
 \bar X(t)\in J_i &\hbox{ for }t\in\left[t_{n}-\frac{|x|}{\bar \alpha_i},t_{n}\right),\vspace{0.2cm}\\ 
  \bar X(t)=0 &\hbox{ for } t\in\left[t_{n}-\frac{|x|}{\bar \alpha_i}-\bar s_0,t_{n}-\frac{|x|}{\bar \alpha_i}\right),\vspace{0.2cm}\\
  \bar X(t)\in J_j &\hbox{ for }t\in\left[t_{n-1},t_{n}-\frac{|x|}{\bar \alpha_i}-\bar s_0\right].
  \ea\right.
  \ee
    Then
\begin{align*}
&u(t_n,x)-w^n(x)\\
=& \inf_{y\in J}\inf_{(X,\alpha)\in \Gamma_{t_{n-1},y}^{t_n,x}}\left[u(t_{n-1}, y)+\int\limits_{t_{n-1}}^{t_n} L(X(\tau),\alpha(\tau))d\tau\right]-S[w^n](x)\\
\le&
u(t_{n-1},X(t_{n-1}))+\int_{t_{n-1}}^{t_{n}} L(X(\tau),\alpha(\tau))d\tau-S[w^n](x)\\
\leq &u(t_{n-1},X(t_{n-1}))-\I[w^{n-1}](\bar X(t_{n-1}))|\\
&+\int_0^{\Dt-\frac{|x|}{\bar \alpha_i}-\bar s_0} L(\bar X(\tau),\bar\alpha(\tau))d\tau-\left(\Delta t- \bar s_0-\frac{|x|}{\bar\alpha_i}\right)L_j(\bar\alpha_j)\\
& +\int\limits_{\Delta t- \bar s_0-\frac{|x|}{\bar\alpha_i}}^{\Delta t-\frac{|x|}{\bar\alpha_i}} L(\bar X(\tau),\bar\alpha(\tau))d\tau-s_0 L_0(\bar \alpha_0)+\int\limits_{\Delta t-\frac{|x|}{\bar\alpha_i}}^{\Dt} L(\bar X(\tau),\bar\alpha(\tau))d\tau-\frac{|x|}{\bar\alpha_i}L_i(\bar\alpha_i).
\end{align*}
Using that $L(.,\alpha)$ is constant along an arc, we deduce that the cost terms cancel. Moreover, using that
\begin{align*}
&u(t_{n-1},X(t_{n-1}))-\I[w^{n-1}](\bar X(t_{n-1}))|\\
=&u(t_{n-1},X(t_{n-1}))-\I[u(t_{n-1},\cdot)](X(t_{n-1}))+(n-1)C\Dx\\
&+\I[u(t_{n-1},\cdot)-(n-1)C\Dx](X(t_{n-1}))+\I[w^{n-1}](\bar X(t_{n-1}))\\
\le& nC\Delta x,
\end{align*}
where we have used Lemma \ref{inter} (iii) to control the first term and Lemma \ref{inter} (i) joint to \eqref{eq:rec} for the last one.
This implies that 
$$u(t_n,x)-w^n(x)\le nC\Delta x.
$$

For the inverse inequality, we invert the whole argument. An additional difficulty comes for the choice of the good control for the $S[w^{n-1}]$ term. We proceed considering a continuous optimal control $\bar \alpha(\cdot)$ for $u(t_n,x)$ in \eqref{eq:dppu}. Without loss of generality we assume that the associated trajectory $\bar X(t)$ is such that
 \be \left\{
 \ba{lll}
 \bar X(t)\in J_i, & &\quad\hbox{ for }t\in\left(\bar t_2,t_n\right],\\
  \bar X(t)=0, &   &\quad\hbox{ for } t\in\left(\bar t_1,\bar t_2\right],\\
  \bar X(t)\in J_j, &   &\quad\hbox{ for }t\in\left[t_{n-1},\bar t_1\right],
  \ea\right.
  \ee

Indeed, we can exclude that an optimal trajectory pass in one other arc or touch multiple times the junction point thanks to the convexity of the functions $L$. In fact, in such case, it would be necessary for an optimal trajectory to pass twice for the same point, i.e.  $X(\tilde t_1)=\tilde  x$ and $X(\tilde  t_2)=\tilde  x$, with $X(t)\ne\tilde x$ for $t\in (\tilde t_1, \tilde t_2)$. This means that since $\dot X(t) = \bar \alpha(t)$, we have that 
$$\int_{\tilde t_1}^{\tilde t_2}\bar \alpha (\tau) d\tau=X(\tilde t_1)-X(\tilde t_2)=0.$$
Then, the average control on $[\tilde t_1, \tilde t_2]$ is zero. Using the strict convexity and the Jensen's inequality, we find that the optimal control $\bar \alpha$ should be zero. This contradicts the definition of $X$.

We can now build a discrete control and an associated trajectory $(\hat \alpha,\hat X)$ for $S[\hat \varphi](x)$ such that 
\begin{equation*}
\hat \alpha_i=\frac{|x|}{t_{n} - \bar t_2}=\frac{1}{t_{n}-\bar t_2}\int^{t_{n}}_{\bar t_2}\bar \alpha_i(\tau) d\tau, \quad \hat s_0=\bar t_2-\bar t_1,
\end{equation*}
$$ \hat\alpha_j=\frac{1}{\bar t_1-t_{n-1}}\int_{t_{n-1}}^{\bar t_1}\bar \alpha_j(\tau) d\tau.$$
 Then, for construction $\hat X(t_{n-1})=\bar X(t_{n-1})$ and  
\begin{align*}
&S[w^{n-1}](x)-u(t_n,x)\\
=&S[w^{n-1}](x)- u(t_{n-1}, y)+\int_{t_{n-1}}^{t_n} L(\bar X(\tau),\bar \alpha(\tau))d\tau\\
\leq &\I[w^{n-1}](\hat X(t_{n-1}))-u(t_{n-1},\bar X(t_{n-1}))\\
&+\left(\left(\Delta t- \bar t_2\right)L_i(\hat\alpha_i)-\int_{\bar t_2}^{\Delta t} L(\bar X(\tau),\bar\alpha(\tau))d\tau\right)\\
& +\left((\bar t_2-\bar t_1) L_0-\int_{\bar t_1}^{\bar t_2} L(\bar X(\tau),\bar\alpha(\tau))d\tau\right)\\
& +\left(\bar t_1 L_j(\hat\alpha_j)-\int_{t_{n-1}}^{\bar t_1} L(\bar X(\tau),\bar\alpha(\tau))d\tau\right).
\end{align*}
Using  Jensen's inequality knowing that the $L$-functions are convex, we get
\begin{align*}
&\bar t_1 L_j(\hat\alpha_j)-\int_{t_{n-1}}^{\bar t_1} L(\bar X(\tau),\bar\alpha(\tau))d\tau\\
=&\bar t_1 L_j\left(\frac{1}{\bar t_1-t_{n-1}}\int_{t_{n-1}}^{\bar t_1}\bar\alpha_j(\tau)d\tau\right)-\int_{t_{n-1}}^{\bar t_1} L_j(\bar \alpha_j(\tau))d\tau\\
\leq& \int_{t_{n-1}}^{\bar t_1} L_j(\bar \alpha_j(\tau))d\tau-\int_{t_{n-1}}^{\bar t_1} L_j(\bar \alpha_j(\tau))d\tau=0
\end{align*}
The two others cost terms can be treat in a similar way. Finally, using that
\begin{multline*}
\I[w^{n-1}](\hat X(t_{n-1}))-u(t_{n-1},\bar X(t_{n-1}))\le \I[w^{n-1}](\hat X(t_{n-1}))\\
-\I[u(t_{n-1}, \cdot)+(n-1)C\Dx](\bar X(t_{n-1}))+ \I[u(t_{n-1}, \cdot)](\bar X(t_{n-1}))\\+(n-1)C\Dx -u(t_{n-1},\bar X(t_{n-1}))\le nC\Dx
\end{multline*}
where we have used Lemma \ref{inter} (i) joint to \eqref{eq:rec} to control the first term and Lemma \ref{inter} (iii) for the last one.
This implies that 
$$w^n(x)-u(t_n,x)\le nC\Delta x
$$
and concludes the proof.
\qed\end{proof}

\subsection{Space dependent Hamiltonians}

We prove an error estimate for stable schemes for which a consistency error estimate holds.

\begin{theorem}\label{teo:bounds2}
Considered $u$ a viscosity solution of \eqref{eq:hjnet}, and $w$ a solution of a scheme for which the results Lemma \ref{monotonicity} (monotonicity),  Prop. \ref{stability} (stability) and a result similar to Prop. \ref{consisrates} (consistency error estimate) hold, then there exists a positive constant $C$ independent from $\Delta t$ and $\Delta x$ such that 
\begin{equation}\label{esterr}
 \sup_{(t,x )\in  \mathcal G^{\Delta }}|u(t,x)-w(t,x)| \leq C T \left(\frac{\E(\Delta t,\Delta x)}{\sqrt{\Delta t}}+\sqrt{\Delta t}\right) + \sup_{x\in J^{\Dx}}|u_0(x)-w(0,x)|;
\end{equation}
being $\E(\Delta t,\Delta x)$ the consistency error of the scheme. 
\end{theorem}

\begin{corollary}
In the specific case of the scheme \eqref{eq:slscheme}, if we assume moreover \eqref{CFLg}, we have
\begin{equation}
 \sup_{(t,x )\in  \mathcal G^{\Delta }}|u(t,x)-w(t,x)| \leq C \left( \sqrt{\Delta t}+\frac{\Delta x^2}{\sqrt{\Delta t^3}}\right).
\end{equation}
\end{corollary}

\begin{proof}
As standard in this kind of proof, we only prove that
\begin{equation}\label{esterr1}
u(t,x)-w(t,x) \leq C  \left(\frac{\E(\Delta t,\Delta x)}{\sqrt{\Delta t}}+\sqrt{\Delta t}\right)+ \sup_{x\in J^{\Dx}}|u_0(x)-w(0,x)|\quad \hbox{in } \mathcal G^{\Delta },
\end{equation}
since the reverse inequality is obtained with small modifications. Assume that $T\le 1$ (the case $T\ge 1$ is obtained by induction). 

For $i\in \{1,\dots,N\}$ and $j\in \mathbb N$, wet set $x_j^i=j\Delta x e_i$. We then define the extension in the continuous space of $w$ by 
$$w_\#(t_n,x)=\I[\hat w(t_n,\cdot)](x).$$
Firstly, we assume that 
$$u_0(x_j^i)\ge w_\#(0,x_j^i)\quad{\rm for\;  all\;}i\in\{0,\dots,N\}{\rm \; and\;}j\in \mathbb N.$$
We define
$$0\leq\mu_0:=\sup_{x\in J}\{|u_0(x)- w_\#(0,x)|\},$$
and we assume that $\mu_0\le K$. For every $\beta, \eta \in (0,1)$ and $\sigma>0$,  we define the auxiliary function, $\hbox{ for }(t,s,x)\in[0,T)\times \{t_n: n=0,\dots, N_T  \}\times J$
$$\psi(t,s,x):=u(t,x)-w_\#(s,x)- \frac{(t-s)^2}{2\eta}-\beta |x|^2-\sigma t, \\ 
$$
Using Proposition \ref{stability}, the inequality $|u(x,t)-u_0(x)| \leq C_T$ (which holds for the continuous solution, see Theorem 2.14 in \cite{imbert2013flux}),  we deduce that  $\psi(t,s,x)\to -\infty$ as $|x|\to +\infty$ and then the function $\psi$ achieves its maximum at some point $(\oot_\beta,\os_\beta,\ox_\beta)$.
In particular, we have
$$\psi(\oot_\beta,\os_\beta,\ox_\beta)\ge \psi(0,0,0)=u_0(0)-w_\#(0,0)\ge 0.$$
We denote by $K$ several positive constants only depending on the Lipschitz constants of $u$.

\noindent{\bf Case 1:} $\ox_\beta \in J_i\setminus \{0\}$.\newline
In this case, we duplicate the space variable by considering, for $\e\in(0,1)$,
\begin{align*}
\psi_1(t,s,x,y)=&u(t,x)-w_\#(s,y)- \frac{(t-s)^2}{2\eta}-\frac{d(x,y)^2}{2\e}-\frac \beta 2( |x|^2+|y|^2) - \sigma t\\
&-\frac \beta 2|x-\ox_\beta|^2-\frac\beta 2 | y-\bar x_\beta|^2-\frac \beta 2|t-\oot_\beta|^2-\frac \beta 2|s-\os_\beta|^2,\\
&\hbox{ for }(t,s,x,y)\in[0,T)\times \{t_n: n=0,\dots, N_T  \}\times J\times J.
\end{align*}
Using Proposition \ref{stability} again, the inequality $|u(x,t)-u_0(x)| \leq C_T$,  and the fact that $u_0$ is Lipschitz continuous, we deduce that  $\psi_1(t,s,x,y)\to -\infty$ as $|x|,|y|\to +\infty$ and then the function $\psi_1$ achieves its maximum at some point $(\oot,\os,\ox,\oy)$, i.e.
\begin{equation*}
  \psi_1(\oot,\os,\ox,\oy)\geq \psi_1(t,s,x,y)\quad \hbox{ for all } (t,x),(s,y)\in[0,T)\times J.
\end{equation*}
It is also easy to show that $(\oot,\os,\ox,\oy)\to(\oot_\beta,\os_\beta,\ox_\beta,\ox_\beta)$ as $\e$ goes to zero and so $\ox,\oy\in J_i\setminus\{0\}$, for $\eps$ small enough. 

\noindent{\bf Step 1. (Basic estimates). } The maximum point of $\psi_1$ satisfy the following estimates:

\be\label{stima2} d(\ox,\oy)\le K \varepsilon, \quad |\oot-\os|\le K \eta. \ee
\be\label{stima1}\beta\left(|\ox|^2+|\bar y|^2\right)\leq {K}, \quad \beta \left(|\bar x-\bar x_\beta|^2+|\bar y-\bar x_\beta|^2+|\bar t-\bar t_\beta|^2+|\bar s-\bar s_\beta|^2\right){\leq  K }\ee

From 
$$\psi_1(\oot,\os,\ox,\oy)\ge \psi_1(\oot_\beta,\os_\beta,\ox_\beta,\ox_\beta) {=} \psi(\oot_\beta,\os_\beta,\ox_\beta)\ge 0,$$
we get, (using $0\geq -(\bar t-\bar s)^2/2\eta-d(\bar x,\bar y)^2/2\varepsilon-\sigma \bar t$)
\begin{align}\label{eq:105}
&\frac  \beta 2 (|\bar x|^2+|\bar y|^2)+\frac \beta 2 \left(|\bar x-\bar x_\beta|^2+|\bar y-\bar x_\beta|^2+|\bar t-\bar t_\beta|^2+|\bar s-\bar s_\beta|^2\right)\nonumber\\
\le& u(\bar t,\bar x)-w_\#(\bar s,\bar y)\le u_0(\bar x)-w_\#(0,\bar y)+K\bar t+K\bar s\le K(1+|\bar x|+|\bar y|)
\end{align}
where we have used Proposition  \ref{stability} (extended to all the points of $J$ thanks to the monotonicity of the interpolation operator Lemma \ref{inter}) and  \cite[Theorem 2.14]{imbert2013flux} for the second inequality and the fact that $T\le 1$ for the last one.
Using Young's inequality, (i.e. the fact that $|\bar x|\leq 1/\beta + \beta/4 |\bar x|^2$ since $(\beta/2|\bar x|-1)^2\geq 0$) \eqref{eq:105} implies in particular that
$$\frac  \beta 2 (|\bar x|^2+|\bar y|^2)\le K\left(1+\frac 2 \beta +\frac \beta 4 (|\bar x|^2+|\bar y|^2)\right).$$
Multiplying by $\beta$ and using the fact that $\beta \le 1$, we finally deduce that
$$\beta|\bar x|,\; \beta |\bar y|\le K.$$
Using \eqref{eq:105} again,  the equation above implies that
\begin{equation*}
\beta  \left(|\bar x-\bar x_\beta|^2+|\bar y-\bar x_\beta|^2+|\bar t-\bar t_\beta|^2+|\bar s-\bar s_\beta|^2\right)\le  K\left(1+\frac 1 \beta\right)
\end{equation*}
and so
$$\beta \left(|\bar x-\bar x_\beta|+|\bar y-\bar x_\beta|+|\bar t-\bar t_\beta|+|\bar s-\bar s_\beta|\right)\le K.$$

From $\psi_1(\oot,\os,\ox,\oy)\geq \psi_1(\oot,\os,\oy,\oy)$ 
we get
\begin{multline} \label{stima5b}
\frac{d(\ox,\oy)^2}{2\varepsilon}\leq u(\oot,\ox)-u(\oot,\oy)+\frac \beta 2(|\bar y|^2- |\ox|^2)+\frac \beta 2(|\bar y- \bar x_\beta|^2-|\bar x-\bar x_\beta|^2)\\
\leq K d(\bar x,\bar y)+\frac \beta 2 (|\bar x|+|\bar y|)d(\bar x,\bar y) + \frac \beta 2(|\bar x-\bar x_\beta|+|\bar y-\bar x_\beta|)d(\bar x,\bar y)
\le K d(\bar x,\bar y)
\end{multline}
which implies the first estimate of \eqref{stima2}. The second bound in \eqref{stima2} is deduced from $\psi(\oot,\os,\ox,\oy)\geq \psi(\os,\os,\ox,\oy)$ in the same way.

If we include the estimate
$$u(\bar t,\bar x)-w_\#(\bar s,\bar y)\le u_0(\bar x)+K\bar t-w_\#(0,\bar y){+ K\bar s}\le K(\mu_0+d(\bar x,\bar y)+{1})\le {K}$$
in the first part of \eqref{eq:105}, we finally deduce \eqref{stima1}.
\bigskip

{\bf{Step 2.} (Viscosity inequalities). }
We claim  that for $\sigma$ large enough, the supremum of $\psi_1$ is achieved for $\oot=0$ or $\os=0$. 
We prove the assertion by contradiction. Suppose $\oot>0$ and $\os>0$.

Using the fact that  $(t,x)\to\psi_1(t,\bar s, x,\bar y)$ has a maximum in $(\bar x,\bar t)$ and that $u$ is a sub solution, we get
\begin{equation}\label{sub1}
 \frac{\bar t-\bar s}{\eta}+\sigma+\beta(\bar t-\bar t_\beta)+ H_i\left(\frac{d(\bar x,\bar y)}{\varepsilon}+\beta |\bar x|+\beta(|\bar x-\bar x_\beta|)\right)\leq 0.
 \end{equation}
Since $\bar s>0$ we know that 
$\psi_1(\bar t, \bar s, \bar x, \bar y)\geq \psi_1(\bar t, \bar s-\Delta t, \bar x, y)$ for a generic $y$ and, by defining $\varphi(s,y)=-\left(\frac{(\bar t-s)^2}{2\eta}+\frac{d(\bar x,y)^2}{2\varepsilon}+\frac{\beta}{2}|y|^2+\frac{\beta}{2}|y-\bar x_\beta|^2+\frac{\beta}{2}|s-\bar s_\beta|^2\right)$,  it implies that for a generic $y$ holds
$$w_\#(\bar s,\bar y)-\varphi(\bar s,\bar y)\leq w_\#(\bar s-\Delta t, y)-\varphi(\bar s-\Delta t, y).$$
In particular, we have that for any  $ z \in J^{\Dx }$
$$w_\#(\bar s,\bar y)-\varphi(\bar s,\bar y)\leq w(\bar s-\Delta t, z)-\varphi(\bar s-\Delta t, z).$$
By the monotonicity of the scheme
and the fact that the scheme commutes by constant, i.e.  $S[ \hat \varphi+C](z)=S[ \hat \varphi](z)+C$ for any constant $C$, choosing $C=w_\#(\bar s,\bar y)-\varphi(\bar s,\bar y)$   we get for any  $ z \in J^{\Dx }$
$$ w(\bar s,z)=S[ \hat  w(\bar s-\Delta t)]( z)\geq S[\hat  \varphi(\bar s-\Delta t)](z) +C.$$
By the monotonicity of the interpolation operator, this implies 
$$ w_\#(\bar s,\bar y)=\I[\hat  w(\bar s,\cdot)](\bar y)\geq \I[S[\hat  \varphi(\bar s-\Delta t)](\cdot)](\bar y) +w_\#(\bar s,\bar y)-\varphi(\bar s,\bar y),$$
Simplifying by $w_\#(\bar s,\bar y)$, we get
$$-\sum_i \phi_i (\bar y) S[ \hat \varphi(\bar s-\Delta t)](y_i) =-\I[S[\hat  \varphi(\bar s-\Delta t)](\cdot)](\bar y) \geq - \varphi(\bar s,\bar y) ,$$
where $\phi_i$ are the basis functions of the interpolation operator (cf. Lemma \ref{inter}). 
Adding and subtracting  $\I[\hat \varphi(\bar s,\cdot)](\bar y)-\I[\hat \varphi(\bar s- \Delta t,\cdot)](\bar y)$ and dividing by $\Delta t$ , we get
\begin{multline*}
\sum_i \phi_i (\bar y) \left(\frac{\varphi(\bar s-\Delta t, y_i)-S[ \hat \varphi(\bar s-\Delta t)]( y_i) }{\Delta t}+\frac{\varphi(\bar s, y_i)-\varphi(\bar s-\Delta t, y_i)}{\Delta t}
\right) \\
\geq \mathcal{O}\left(\frac{(\Delta x) ^2}{\eps}\right),
\end{multline*}
where we have used the fact that $\varphi_{xx}=\mathcal O(\frac 1\e)$ joint to Lemma \ref{inter}.
We observe that $\frac{\varphi(\bar s, y_i)-\varphi(\bar s-\Delta t,y_i)}{\Delta t}=\varphi_s(\bar s, y_i)+\mathcal O(\Delta t/\eta)$, then, using the consistency result -- Prop \ref{consisrates},  we arrive to
\begin{multline*}
\sum \phi_i (\bar y) \left(-\varphi_s(\bar s, y_i)+H_i(\varphi_x(\bar s-\Delta t, y_i))\right)\\
\geq  \mathcal O\left(\frac{\Delta t}{\eta}+\frac{(\Delta x) ^2}{\eps} \right)+\frac{\E(\Delta t,\Delta x)}{\e}.
\end{multline*}
By the regularity of $\varphi$ and $H$ (Lipschitz continuous) and the interpolation error for Lipschitz function (see Lemma \ref{inter}),  there exists a positive constant $K$ such that
\begin{equation}\label{sub2}
\varphi_s(\bar s, \bar y)+H_i(\varphi_x(\bar s-\Delta t, \bar y))\geq - K\left(\frac{\Delta t}{\eta}+\frac{(\Delta x)^2}{\varepsilon}\right)+\frac{\E(\Delta t,\Delta x)}{\e}.
\end{equation}
We subtract \eqref{sub2} to \eqref{sub1} and expliciting $\varphi$,  obtaining
\begin{multline*}\sigma +\beta(\bar s-\bar s_\beta)+\beta(\bar t-\bar t_\beta)+  H_i\left(\frac{d(\bar x,\bar y)}{\varepsilon}+\beta| \bar x|+\beta(|\bar x-\bar x_\beta|)\right)\\
-H_i\left(\frac{d(\bar x,\bar y)}{\varepsilon}-\beta |\bar y|-\beta(|\bar y-\bar x_\beta|)\right)\leq   K\left(\frac{\Delta t}{\eta}+\frac{(\Delta x)^2}{\varepsilon}\right)+\frac{\E(\Delta t,\Delta x)}{\e}.
\end{multline*}
Then, using that $H_i$ is Lipschitz continuous and the basic estimates of the Step 1, we arrive  to
\begin{equation}\label{eq:3}
 \sigma <
K\sqrt{\beta}+ K\left(\frac{\Delta t}{\eta}+\frac{(\Delta x)^2}{\varepsilon} \right)+\frac{\E(\Delta t,\Delta x)}{\e}=:\sigma^*.
\end{equation}
Therefore, we have that for a $\sigma\ge\sigma^*$ at least one between $\bar t$ and $\bar s$ is equal to zero.
\bigskip

{\bf{Step 3.} (Conclusion). }
If $\overline t=0$ we have
 \begin{align*} \psi_1(0,\bar s,\bar x,\bar y)\leq &u_0(\overline x)-w_\#(\overline s, \overline  y)\leq u_0(\overline x)-u_0(\overline y)+C\overline s +\mu_0\\
\leq &K\e+K\eta+\mu_0.
\end{align*}
A similar argument applies if $\bar s=0$.

Taking $\sigma=\sigma^*$, we get
\begin{align*}
&u(t,x)-w_{\#}(t,x)-\frac\beta2\left(|x|^2+|y|^2+|x-\bar x_\beta|^2+|y-\bar x_\beta|^2+|t-\bar t_\beta|^2+|s-\bar s_\beta|^2\right)\\
&-\left(K\sqrt{\beta}+ K\left(\frac{\Delta t}{\eta}+\frac{(\Delta x)^2}{\varepsilon} \right)+\frac{\E(\Delta t,\Delta x)}{\e}\right)T\\
\le&K\e+K\eta+\mu_0.
\end{align*}
Sending $\beta\to0$ and choosing $\e=\eta=\sqrt {\Dt}$, we get the desired estimate.

\bigskip

\noindent{\bf Case 2: } $\ox_\beta =0$.\newline
Firstly we observe that assuming 
\be \label{hcont}
\sigma > K\sqrt{\beta}+ K\left(\frac{\E(\Delta t,\Delta x)}{\e}+\frac{\Delta t}{\varepsilon}+\frac{\Delta x}{\eps}\right)
\ee
(which is compatible with $\sigma>\sigma^*$) then, there exists a $\bar A\in \mathbb{R}$ such that 
\begin{equation}\label{Ahyp}
 \frac{\bar s_\beta-\bar t_\beta}{\eta} -K\left(\frac{\E(\Delta t,\Delta x)}{\e}+\frac{\Delta t}{\varepsilon}\right)+K\sqrt{\beta}>\bar A > \frac{\bar s_\beta-\bar t_\beta}{\eta} -\sigma+K\sqrt{\beta}.
\end{equation} 
Using the fact that  $(t,x)\to\psi(t,\bar s_\beta, x)$ has a maximum in $(\bar t_\beta,\bar x_\beta)$ and that $u$ is a sub solution, we get
\begin{equation}\label{sub3}
 \frac{\bar t_\beta-\bar s_\beta}{\eta} +\sigma+ F_A\left(\partial_x\varphi(\bar t_\beta,0)\right)\leq 0,
 \end{equation}
 with $\varphi(t,x)=w_\#(\bar s_\beta,x)+\frac{(t-\bar s_\beta)^2}{2\eta}+\beta |x|^2+\sigma t$
and from \eqref{sub3} and \eqref{Ahyp}, 
\begin{equation}\label{ght}
\bar A>  F_A\left(\partial_x\varphi(\bar t_\beta , 0)\right).
\end{equation} 
We use \eqref{ght} the definition of $F_A$, and the coercivity of the Hamiltonians to obtain the existence of  values $\lambda_i$ such that 
\begin{equation}\label{Hbound2}
H_i(\lambda_i)=H_i^+(\lambda_i)=\bar A
\end{equation} 
(cf. Fig. \ref{figH}) that will be useful in the following of the proof.

\begin{figure}[h!]
\begin{center}
\includegraphics[height=6.5cm]{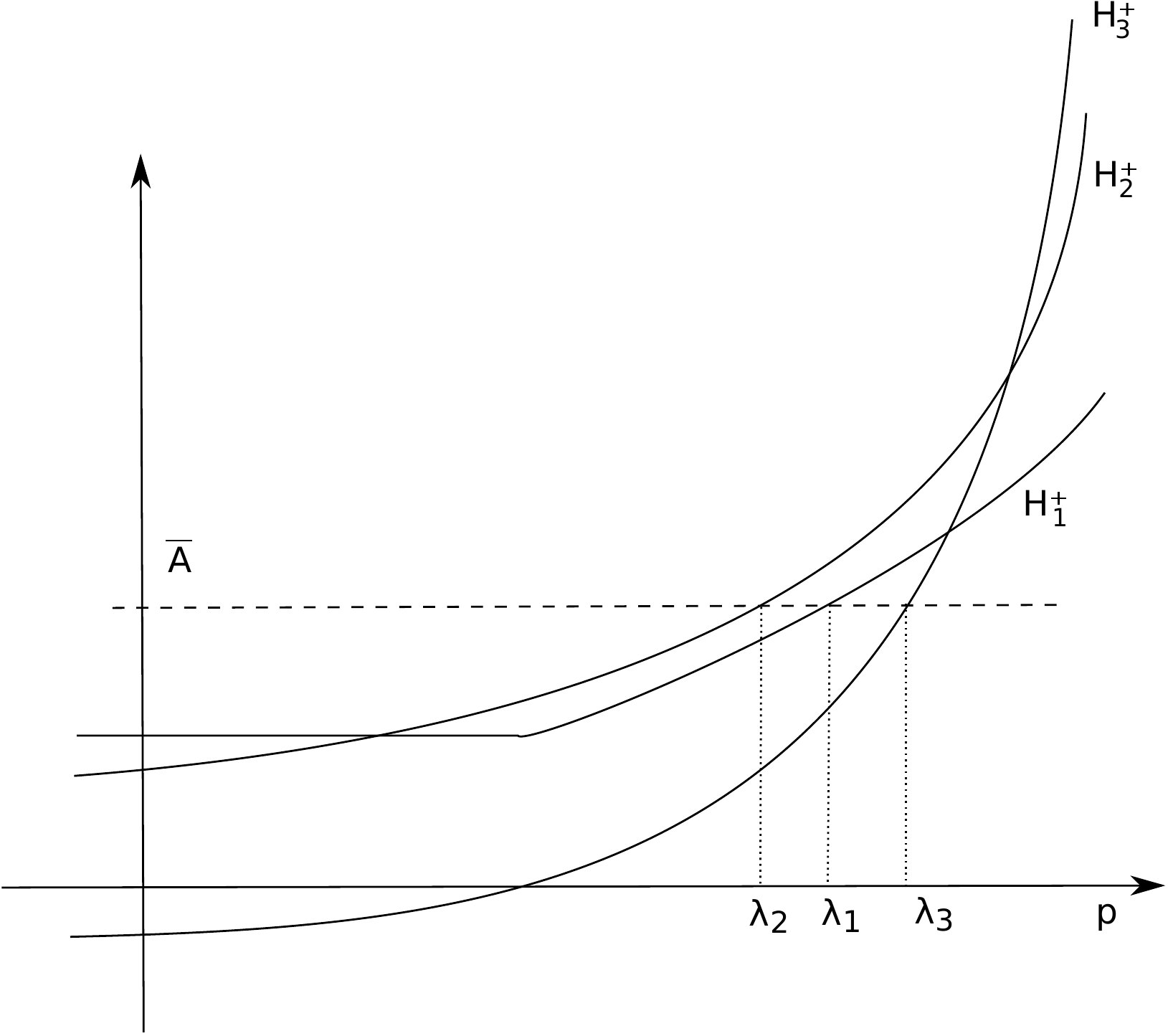} 
\caption{An example of $H^+_i$ functions.} \label{figH}
\end{center}
\end{figure}

Now we pass to identify the right test function to treat this case. We duplicate the space variable differently than in Case 1. We consider, for $\e\in(0,1)$,
\begin{align*}
\psi_2(t,s,x,y)&=u(t,x)-w_\#(s,y)- \frac{(t-s)^2}{2\eta}-\frac{d(x,y)^2}{2\e}- \frac \beta 2( |x|^2+|y|^2)\\
& - \sigma t-\left(h(x)+h(y)\right)-\frac \beta 2(t-\oot_\beta)^2-\frac \beta 2(s-\os_\beta)^2,\\
&\hbox{ for }(t,s,x,y)\in[0,T)\times \{t_n: n=0,\dots, N_T  \}\times J\times J.
\end{align*}
where $h(x)=\lambda_i x$ if $x\in J_i$ and the $\lambda_i$ are defined in \eqref{Hbound2}.

We denote by $(\oot,\os,\ox,\oy)$ the maximum point of $\Psi_2$ (we keep the same notation than the previous case, but they are possibly different points). We remark $(\oot,\os,\ox,\oy)\to(\oot_\beta,\os_\beta,\ox_\beta,\ox_\beta)$ as $\e\to 0$. 
\bigskip

{\bf{Step 2.} (Viscosity inequalities). }
We claim  that for $\sigma$ large enough, the supremum of $\psi_1$ is achieved for $\oot=0$ or $\os=0$. 
We prove the assertion by contradiction. Suppose $\oot>0$ and $\os>0$.
We can have different scenarios: if $\ox$ and $\oy$ belong to the same arc (junction point excluded) the case is included in Case 1. If instead   $\bar x\in J_i\setminus\{0\}$, $\bar y\in J_j$ ($\bar x$ and $\bar y$ belong to different arcs), we can repeat the same argument to obtain  \eqref{sub1} with the test function $\psi_2$. We have:
$$\frac{\bar t- \bar s}{\eta}+\beta(\bar t-\bar t_\beta)+\sigma + H_i\left(\frac{d(\bar x,\bar y)}{\varepsilon}+2\beta |\bar x|+\lambda_i\right)\leq 0.$$
Observing that the argument inside the Hamiltonian is bigger than $\lambda_i$, we use \eqref{Hbound2}  arriving to 
$$0\geq\frac{\bar t- \bar s}{\eta}+\beta(\bar t-\bar t_\beta)+\sigma + H_i^+(\lambda_i)= \frac{\bar t_\beta- \bar s_\beta}{\eta}+\sigma + \bar A +K\sqrt{\beta},$$
which contradicts \eqref{Ahyp}. Then, this case cannot appear.

We pass to the last case to consider: $\bar x=0$, $\bar y\in J_i\setminus\{0\}$. First of all, we notice that the \emph{basic estimates} \eqref{stima2}-\eqref{stima1} are still valid for $(\oot,\os,\ox,\oy)$ maximum point of $\psi_2$ since the added terms $h(x)$, $h(y)$ are easily included in the other linear elements of the estimates.

In this case, the difficulty comes comparing two Hamiltonians evaluated, respectively, on the junction point and on one arc. 
 Using the subsolution property with the test function $\psi_2$, we have as first equation:
\begin{equation}\label{eq1}
\frac{\bar t-\bar s}{\eta}+\beta(\bar t-\bar t_\beta)+\sigma + F_A\left(\frac{-|\bar y|}{\varepsilon}+\lambda_i\right)\leq  0,
\end{equation}
where 
$$F_A\left(\frac{-|\bar y|}{\varepsilon}+\lambda_i\right)=\max\left(A,\, \max_j\left(H_j^-\left(\frac{-|\bar y|}{\varepsilon}+\lambda_i\right)\right)\right).$$
From the definition of $F_A$ it is also valid 
\begin{equation}\label{eq2}
\frac{\bar t-\bar s}{\eta}+\beta(\bar t-\bar t_\beta)+\sigma + H^-_i\left(\frac{-|\bar y|}{\varepsilon}+\lambda_i\right)\leq  0.
\end{equation}
Since $\bar y\in J_i\setminus\{0\}$ with the same argument to obtain \eqref{sub2} (but for the test function $\psi_2$) and using the consistency result, we have
$$\frac{\bar t-\bar s}{\eta}+\beta(\bar s-\bar s_\beta)+H_i\left(\frac{-|\bar y|}{\varepsilon}-2\beta \bar y+\lambda_i\right)\geq K\left(\frac{\E(\Delta t,\Delta x)}{\e}+\frac{\Delta t}{\eta}+\frac{(\Delta x)^2}{\eps}\right).$$
 Now recalling that  $H^+(\lambda_i)=\bar A$,
 \begin{align*}
&\frac{\bar t-\bar s}{\eta}+\beta(\bar s-\bar s_\beta) +H^+_i\left(\frac{-|\bar y|}{\varepsilon}-2\beta \bar y+\lambda_i\right)-K\left(\frac{\E(\Delta t,\Delta x)}{\e}+\frac{\Delta t}{\eta}+\frac{(\Delta x)^2}{\eps}\right)\\
 \leq&  \frac{\bar t-\bar s}{\eta}+\beta(\bar s-\bar s_\beta)+H^+_i\left(\lambda_i\right)-K\left(\frac{\E(\Delta t,\Delta x)}{\e}+\frac{\Delta t}{\eta}+\frac{(\Delta x)^2}{\eps}\right)\\
 \leq&\frac{\bar t_\beta-\bar s_\beta}{\eta}+K\sqrt{\beta}+\bar A-K\left(\frac{\E(\Delta t,\Delta x)}{\e}+\frac{\Delta t}{\eta}+\frac{(\Delta x)^2}{\eps}\right)\\
 &<0
 \end{align*}
for $\e$ small enough,  where we used $\beta(\bar s-\bar s_\beta)\leq K\sqrt \beta$ (basic estimates). We can claim that 
 \begin{equation}\label{eq3}
 \frac{\bar t-\bar s}{\eta}+\beta(\bar s-\bar s_\beta)+H_i^-\left(\frac{-|\bar y|}{\varepsilon}-2\beta \bar y+\lambda_i\right)\geq K\left(\frac{\E(\Delta t,\Delta x)}{\e}+\frac{\Delta t}{\eta}+\frac{(\Delta x)^2}{\eps}\right).
 \end{equation}
We can finally subtract \eqref{eq3} to \eqref{eq1}, obtaining the desired estimate on $\sigma$
 \begin{equation} \sigma \leq K\sqrt{\beta}+ K\left(\frac{\E(\Delta t,\Delta x)}{\e}+\frac{\Delta t}{\varepsilon}+\frac{\Delta x}{\eps}\right):=\sigma^*.
\end{equation}
In this case we obtain a contradiction with \eqref{hcont}. Then, We have that assuming $\sigma>\sigma^*$ at least one between $\oot$ and $\os$ is equal to zero.

\medskip

{\bf Step 3. (Conclusion).}
We obtain the same estimate as in Case 1.

\bigskip

It just remains to prove the general case (for which we do not assume that $u^0(x)\ge w_\#(0,x),\forall x\in J^\Dx$). Remarking that $\bar u=u+\mu_1$ with $\mu_1=\sup_{x\in J^\Dx}(w_\#(0,x)-u_0(x))$ is a solution of the same equation of $u$ but satisfying $\bar u(0,x)\ge w_\#(0,x), \forall x\in J^\Dx$, we deduce that $\bar u$ satisfies 
\begin{align*}
&\sup_{(t,x )\in  \mathcal G^{\Delta }}(u(t,x)+\mu_1-w(t,x))\\
 \leq &C \left(\frac{\E(\Delta t,\Delta x)}{\sqrt{\Delta t}}+\sqrt{\Delta t}\right) + \sup_{x\in J^{\Dx}}|u_0(x)+\mu_1-w(0,x)|.
\end{align*}
which implies \eqref{esterr1} and ends the prooof of the Theorem.

\hfill\qed\end{proof}

\section{Application to traffic flows models}\label{Sect:traffic}
Traffic flow models aim to understand the motion of agents in structures such as highways or roads and to develop optimal networks avoiding congestions. Typically in these models, differently from \emph{crowd motion models} \cite{cristiani2011multiscale}, the dynamic is described on a finite number of one-dimensional pathways (e.g., travel lanes) connected by junction  points.
In a microscopic model, the behavior of every single agent is considered, whereas a \emph{macroscopic} model considers the density of the agents. The  relation between  micro and  macroscopic scale is still an active subject of research.

In the literature, most of the macroscopic models describe the evolution of the density of cars $\rho:J\times [0,T]\rightarrow [0,\rho_{\max}]$  through a system of conservation laws, see for instance  \cite{garavello2016models,camilli2016discrete}. Recently, a different framework based on  HJ has been proposed. The two models are connecting by duality properties. Following \cite{imbert2013hamilton}, we introduce  a \emph{cumulative density function}
\begin{equation}\label{cum}
u(x,t)=g(t)+\frac{1}{\gamma_i}\int_0^x \rho(y,t)dy, \quad \hbox{ for }x\in J_i
\end{equation}
where $\gamma_i$ are constant parameters modeling the incoming/outgoing fluxes for each arc $J_i$ at the junction points. The function $g(t)$ has to be determined with respect to the conservation of the total density and the initial conditions.

In the same work the authors show that considering $\rho$ solution of a conservation law with flux $f(\rho)$, $u(x,t)$ defined in \eqref{cum} is formally the solution of \eqref{eq:hjnet} with the Hamiltonian as

\begin{equation}
H_i(p)=\left\{
\begin{array}{ll}
-\frac{1}{\gamma_i}f(\gamma_i p) & \hbox{ for $p\geq 0$ ('incoming' edges)},\\
-\frac{1}{\gamma_i}f(-\gamma_i p) & \hbox{ for $p<0$ ('outgoing' ones) }.
\end{array}\right.
\end{equation}

In \cite{forcadel2015homogenization}, a similar macroscopic model of the form \eqref{eq:hjnet} is derived  from a microscopic model by an  homogenization procedure. The microscopic model is based on a  system of ordinary differential equations describing the  flow of each single agent. 
Generally the $f$-functions are often called \emph{fundamental diagram} (cf. for various examples \cite{treiber2013traffic}), and they establish a relation between speed and density of the agents. 

\subsection{Network framework and boundary conditions}

In realistic situations, traffic flows are defined on finite networks. We briefly extend our framework in the case of networks: we call $\N$ a network composed by $N$ edges $E_i$ isometric to a real interval $[a_i,b_i]$ or, in the unbounded case to $[a_i,\infty)$, and a collection of $M$ points $V_i$, called nodes, such that
$$ \N:=\bigcup_{i=1,...,N}E_i,\quad \V:=\{V_i, \;i=1,...,M\} $$
and for any $ i,j \in\{ 1,...,M\},\,i\neq j,$ if  $E_i\cap E_j\neq  \emptyset$,  then, $E_i\cap E_j=V_l,$ with $ V_l \in  \V$.
Clearly, a junction can be represented by a network $\N$. Some details explaining how to extend the results obtained on a junction to a general network are contained in \cite[Appendix B]{imbert2013flux}, intuitively since the complexity of the dynamic can be described locally on one node of the network, the presence of multiple nodes does not drastically change the study.

Some conditions on in-flow, and out-flow boundaries need to be defined. We call  $\B\subset \V$ the set that contains the \emph{boundary nodes} of the network. We also assume that these nodes are simply connected, i.e. we assume that
$$ \hbox{if } V_i\in \B, \hbox{ there exists one and only one }j\in\{1,...,N\} \hbox{ s.t. }V_i\cap E_j\neq \emptyset. $$
Hence, we define two possible types of boundary conditions on $\B$. Let $\B_N,\,\B_D$ be two sets such that $\B=\B_N\cup\B_D$, $\B_N\cap\B_D=\emptyset$ and let us consider
\begin{equation}\left\{
\begin{array}{ll}\label{BC}
   u_x(t,x)=f_N(t), & \hbox{if }x\in\B_N, \\
   u(t,x)=f_B(t), & \hbox{if }x\in\B_D.
\end{array}\right.
\end{equation}
 The first equation corresponds to Neumann-like boundary conditions and it  model traffic fluxes entering in the network. The second one  is a Dirichlet-like condition, and it represents the cost to enter   from inflow nodes.\\
We consider   the following problem, strictly related to \eqref{eq:hjnet}:
\begin{equation}\label{eq:hjnet2}
\left\{
\begin{array}{ll}
 \partial_t u(t,x)+H_i(u_x(t,x))=0 & \hbox{ in } (0,T) \times E_i\setminus \V,\\
\partial_t u(t,x)+F_A(u_x(t,x))=0 & \hbox{ in }  (0,T) \times \V\setminus\B.
\end{array}
\right.
\end{equation} 
provided with \eqref{BC}.
The theory discussed for a junction can also be extended to \eqref{eq:hjnet2}-\eqref{BC}. Some hints can be found in \cite{imbert2013flux}.

\section{Numerical tests}\label{Sect:tests}

In this section, we develop some tests showing  the convergence and the efficiency of the scheme. Firstly, we consider two simple tests to verify the convergence error estimate proved in Section \ref{sect:bounds} numerically. Then,  we focus on a complex case coming from traffic flows in  a more realistic  scenario. 

\paragraph{{\bf Test 1}}
We consider a basic network composed by two edges connecting the nodes $(-1,0)$ and $(1,0)$ with a junction in $(0,0)$. This case can be  seen as an 1D problem in $\Omega=\Omega_1\cup\Omega_2=[-1,0]\cup[0,1]=[-1,1]$ with a discontinuity on the Hamiltonian at the origin. \\
We consider the following Hamiltonian on $\Omega$: 
\begin{equation}\label{Htest1}
H(x,p)=\left\{
\begin{array}{lll}
\frac{p^2}{2}-\frac{1}{2}, & &x\in \Omega_1\\
\frac{p^2}{2}-1, & &x\in \Omega_2.\\
\end{array}\right.
\end{equation}
This example has been used as a benchmark also in \cite{imbert2015error}.
Using the Legendre transform, we rewrite \eqref{Htest1} as  
\begin{equation}
H(x,p)=\left\{
\begin{array}{lll}
\displaystyle\max_{\alpha\in\R}\left(\alpha_1 p-\frac{\alpha_1^2}{2}\right)-\frac{1}{2}, & &x\in \Omega_1\\
\displaystyle\max_{\alpha\in\R}\left(\alpha_2 p-\frac{\alpha_2^2}{2}\right)-1, & &x\in \Omega_2,\\
\end{array}\right.
\end{equation}
We chose the initial condition
\begin{equation}
u_0(x)=\sin(\pi|x|),
\end{equation}
and we impose  Dirichlet boundary conditions $u(t,-1)=u(t,1)=0$.
\begin{figure}[t!]
\begin{center}
\includegraphics[height=5.5cm]{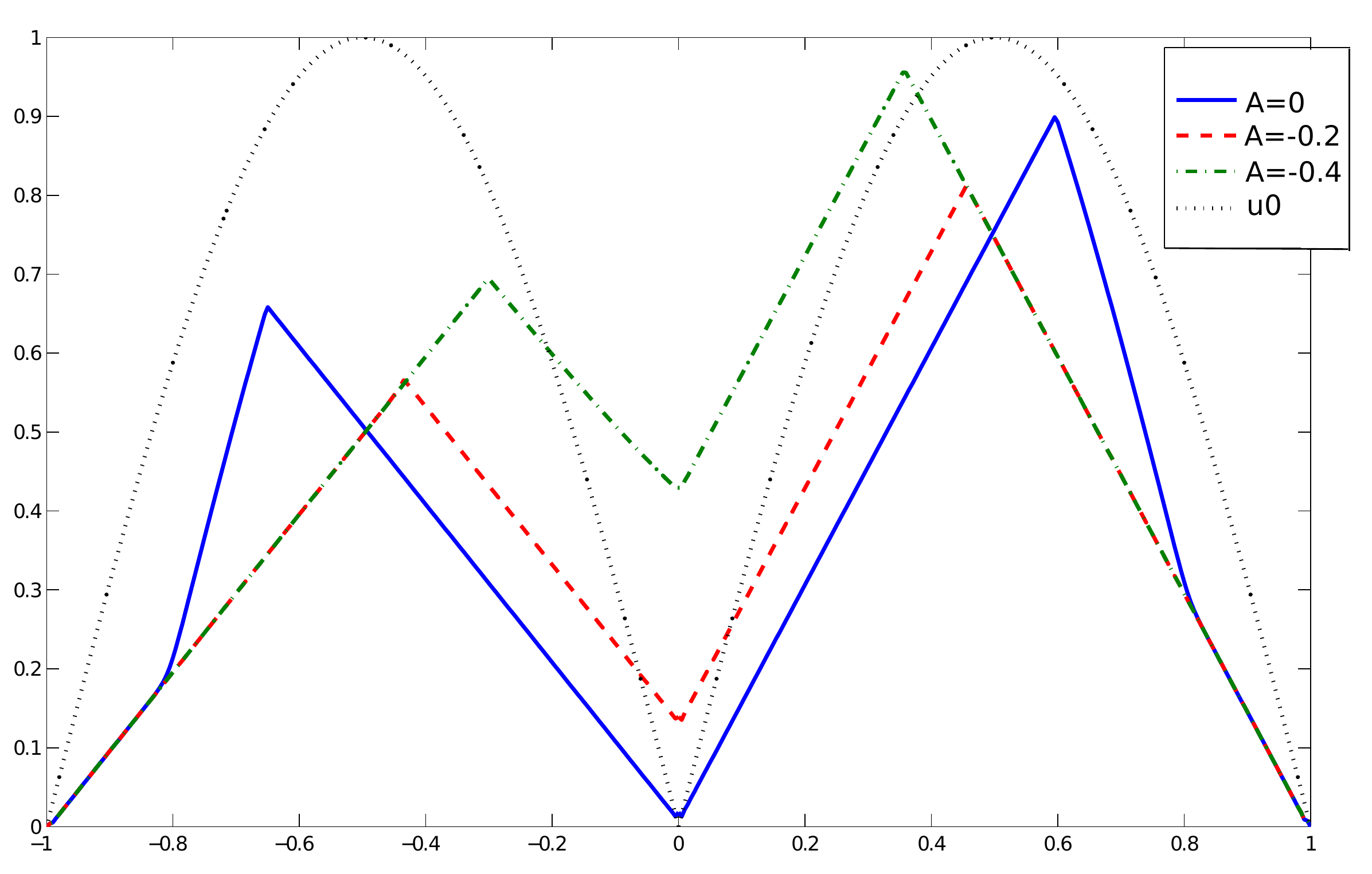} 
\includegraphics[height=5.5cm]{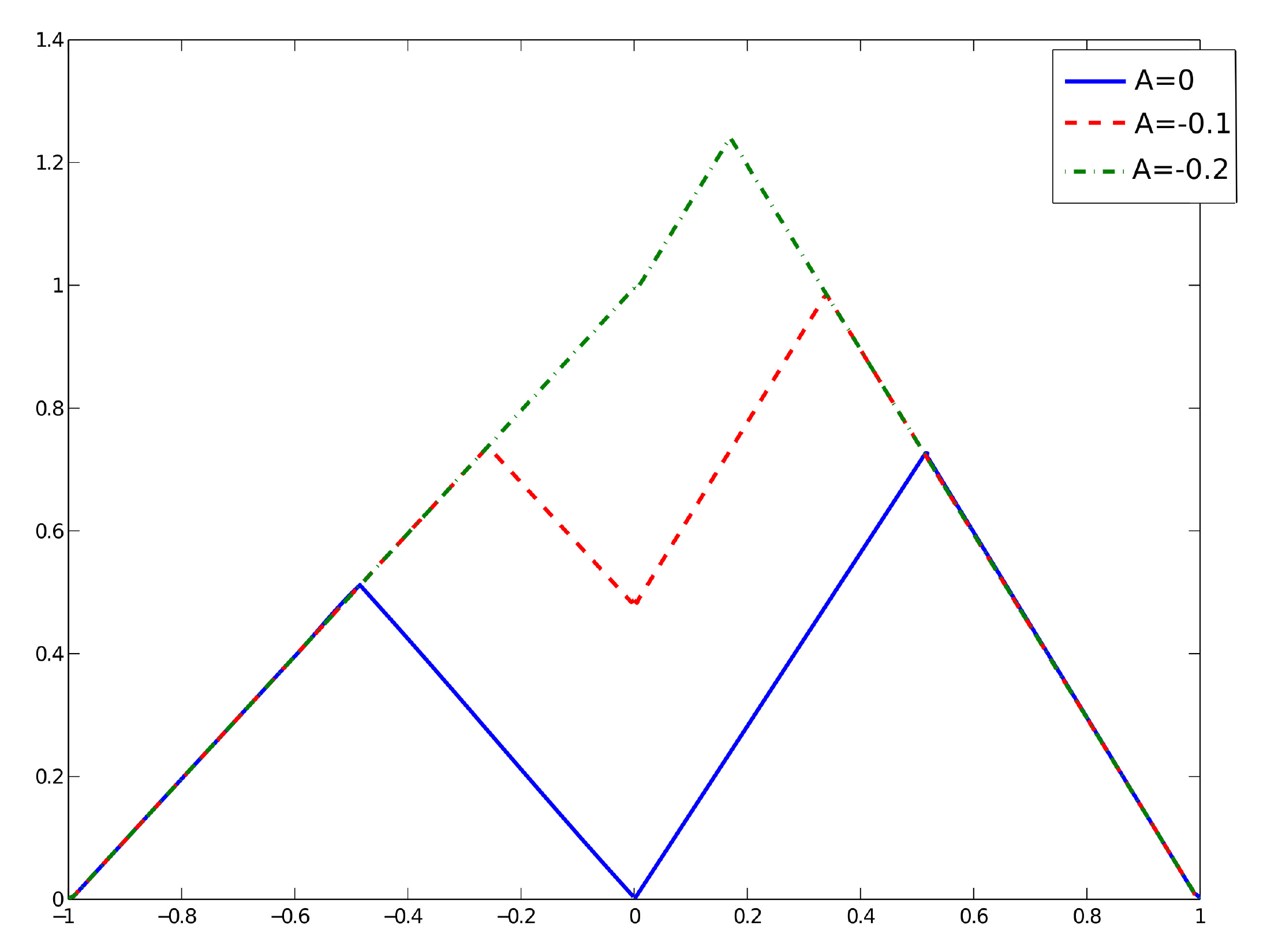}
\caption{Initial solution and numerical solution at time $t=0.2$ (above) with various choices of the parameter $A$. Solution (stationary) at time $t=2$ (below) with various choices of the parameter $A$.} \label{figA}
\end{center}
\end{figure}

In Fig. \ref{figA}, we show the numerical solution at time $t=0.1$ and $t=2$ for different choices of the parameter $A$ on the junction point $x=(0,0)$. We can observe as the asymmetry of the Hamiltonian with respect to the origin induces an asymmetric behavior of the solution.
 We can also observe how the choice of parameter $A$ influences \emph {globally} the value function of the problem. In fact, when $A=0$ the optimal control in $x=0$ is simply $\alpha_0=0$ that corresponds to a zero cost, and since $u_0(0)=0$, the solution $u(t, 0)=0$ for each $t \in [0, T].$
 In the case of $A<0$ the situation is different: the control $ \alpha_0=0$  \emph{does not correspond} to a null cost.
A trajectory, which remains on the junction point, entails a cost. 
 Furthermore, we observe that for values of $|A|$ sufficiently large, the stationary solution  does not change anymore. 
 This is due to the fact that remaining in the junction point is no more a convenient choice.

In the case of $A=0$ and $A=-0.2$, we show the convergence rates. In absence of an analytic exact solution,  we compare the approximated solution $w(T,x)$ with an  approximation $u(T,x)$ obtained on a very fine grid with $\Delta x=10^{-4}$ and $\Delta t=\Delta x$.
We evaluate the error with respect to the uniform discrete norm defined as following
\begin{equation}\label{err}
E_\infty^{\Delta}:=\underset{x \in J^{\Delta x}}{\max}(|w(T,x)- u(T,x)|).
\end{equation}

\begin{figure}[t!]
\begin{center}
\includegraphics[height=5cm]{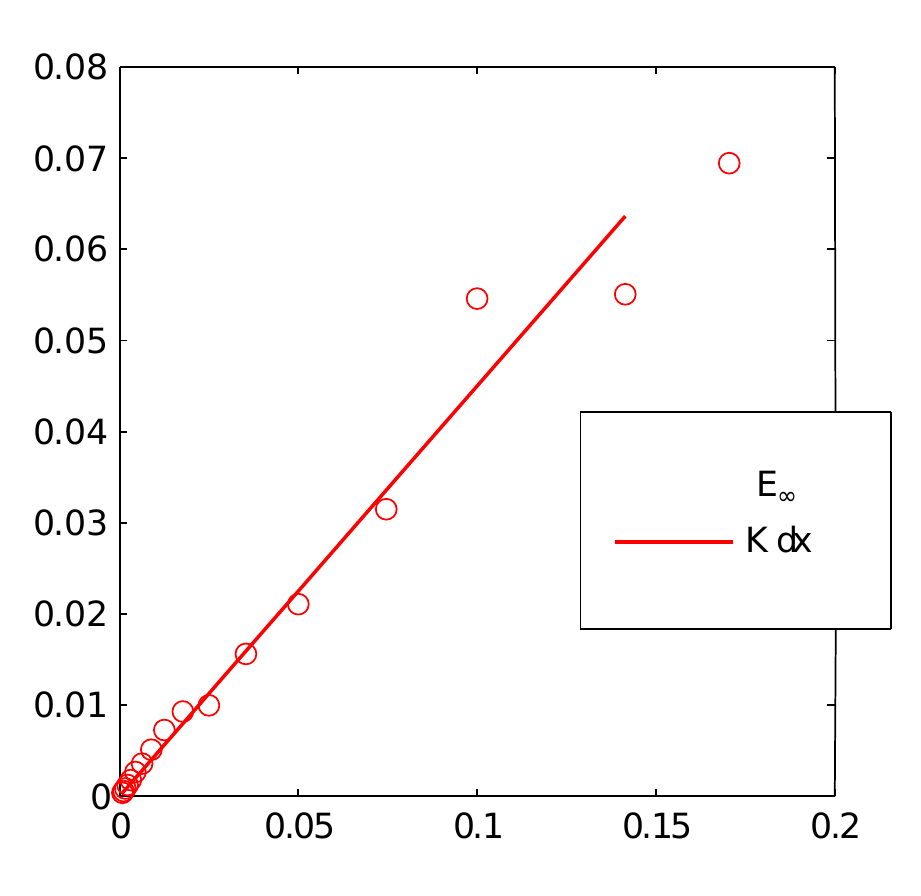}
\includegraphics[height=5cm]{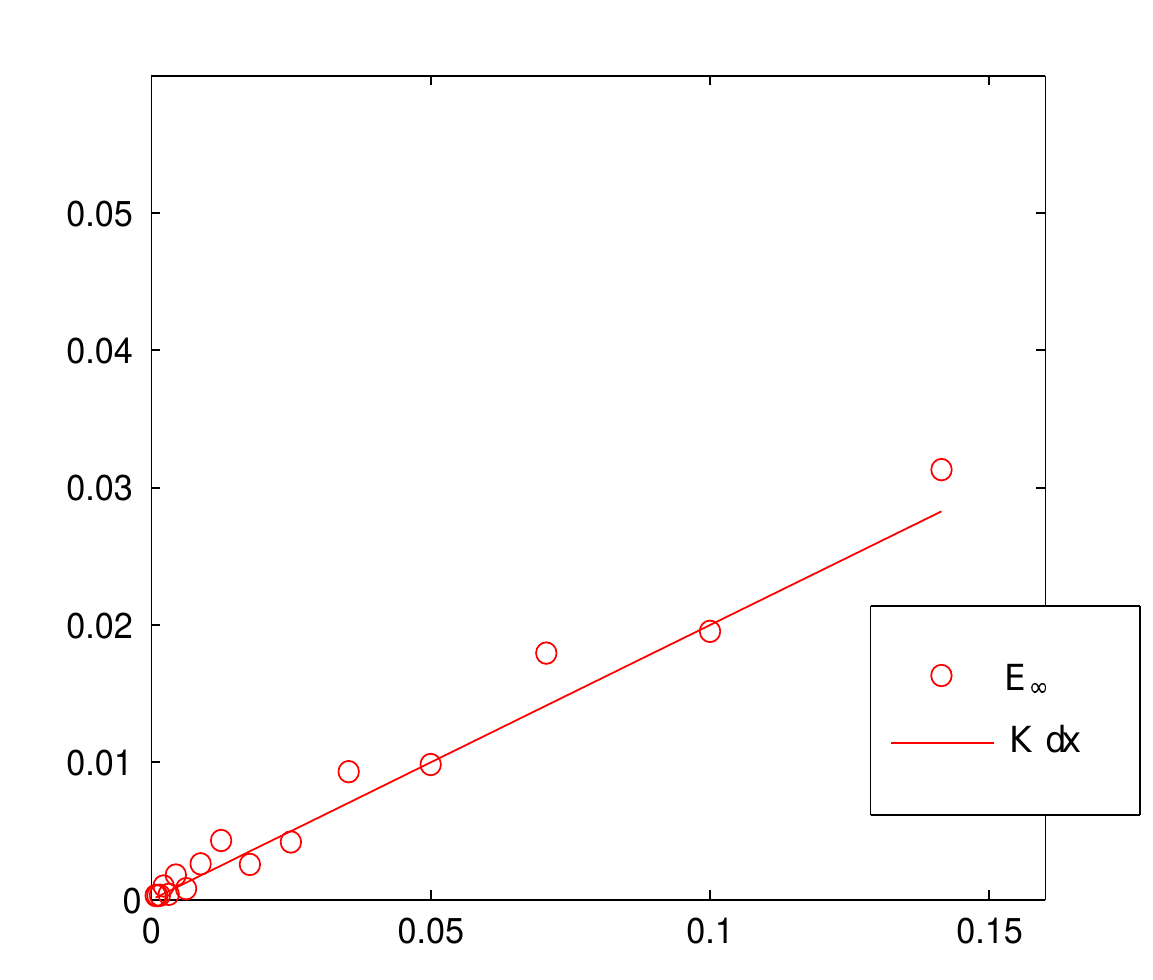}
\caption{Graphic of $E_\infty^{\Delta}$($\circ$) with respect the space step, together with the line $K \Delta x$. Left $A=0$,  with   $K=4.5$, right $A=-0.2$,  with   $K=2$.} \label{figerr2}
\end{center}
\end{figure}

We show the results in Figure \ref{figerr2} for $T=0.2$ and $\Delta t=2.5\Delta x$. We observe in the case $A=0$ a  linear decay of the $E^{\Delta}_\infty$ error, in particular, the $E^{\Delta}_\infty$ errors fit with a linear regression curve of ratio $K_1=4.5$. We also observe the same convergence order in the case $A=-0.2$, even though the linear convergence has a smaller ratio around $K_1=2$. We underline, by Theorem \ref{teo:bounds1}, that we aspect a convergence of order $1$ \emph{independently by the choice of $\Dt$}. This appears confirmed by the test. 

\paragraph{{\bf Test 2}} We consider still a simple junction network composed  by  three edges connecting the nodes $(0,1)$, $(-1,-1)$, $(1,-1)$ with the junction point placed in $(0,0)$. 
\begin{figure}[t!]
\begin{center}
\includegraphics[height=5cm]{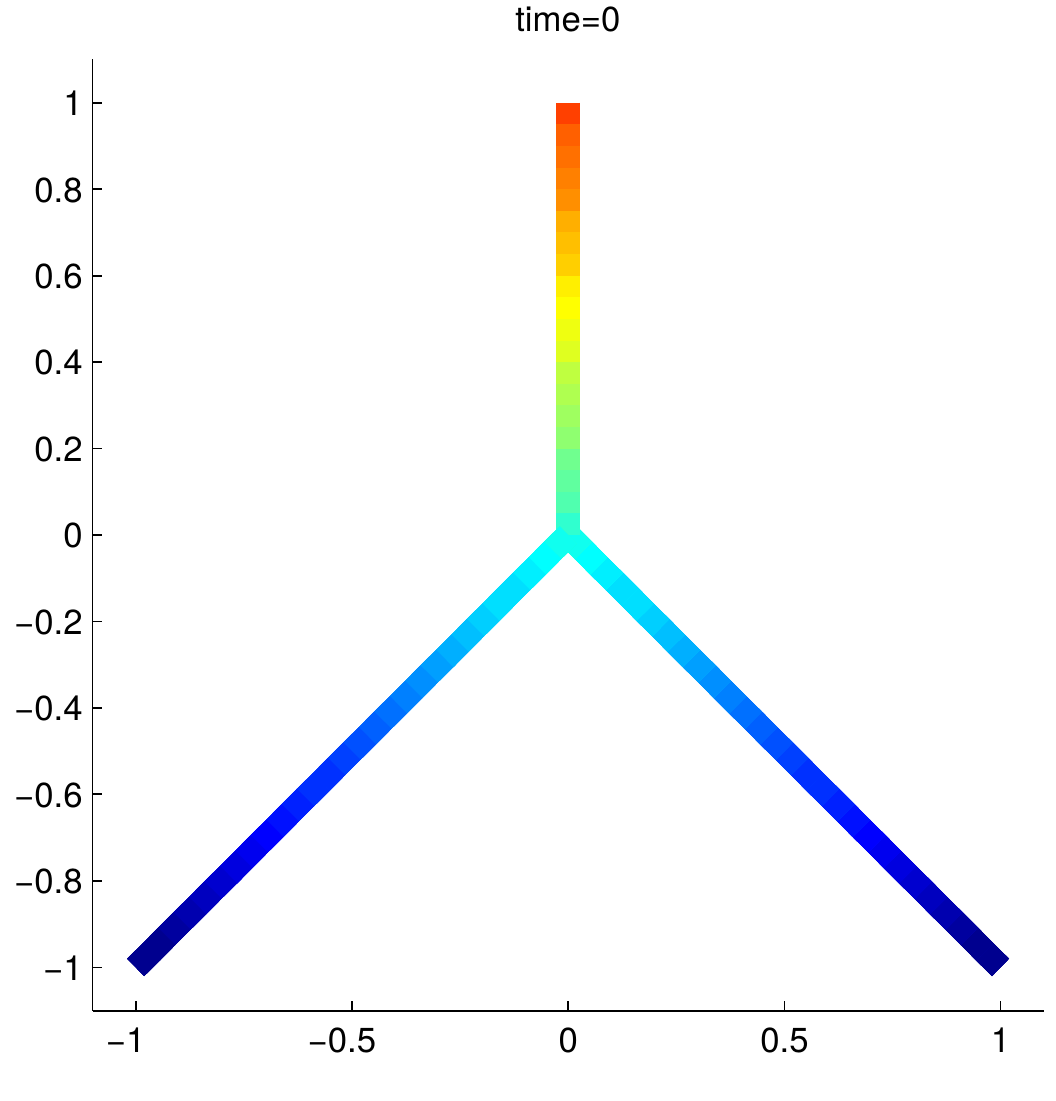}
\includegraphics[height=5cm]{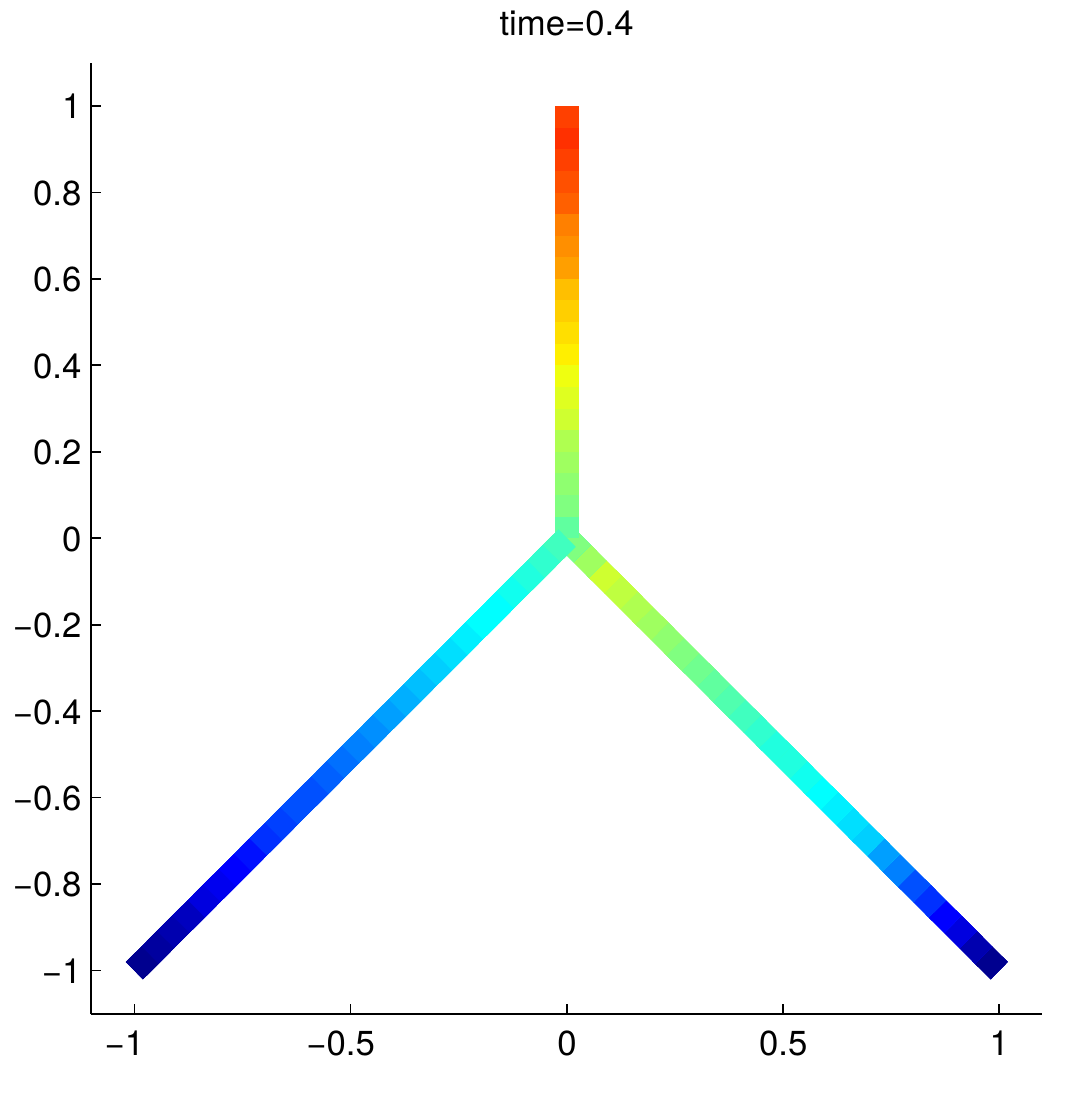}
\includegraphics[height=5cm]{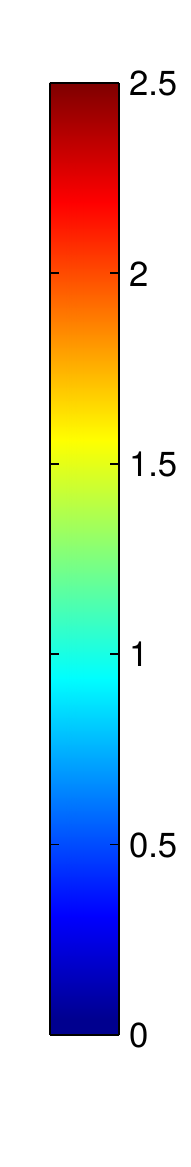}\\
\includegraphics[height=5.2cm]{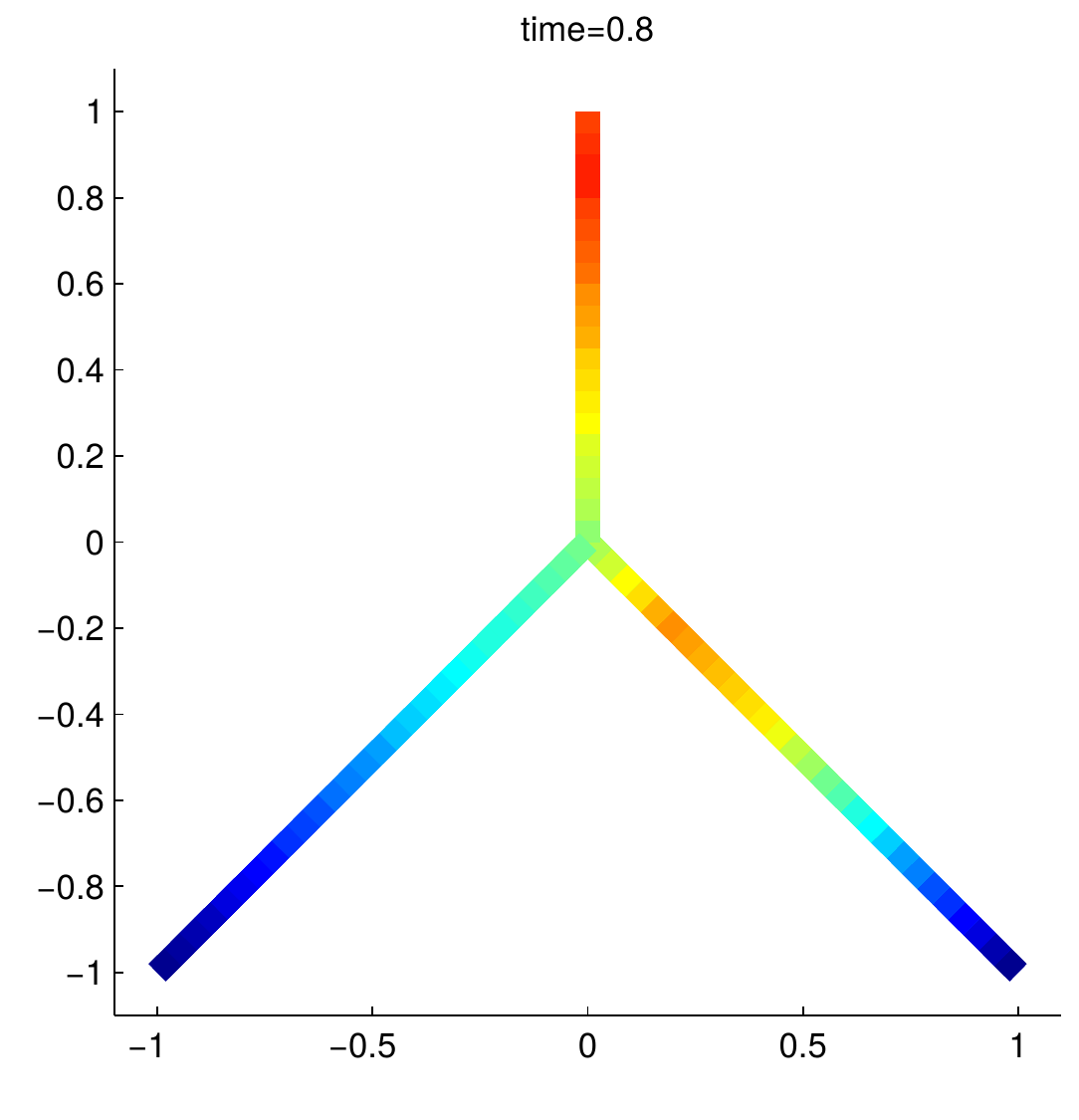}
\includegraphics[height=5.2cm]{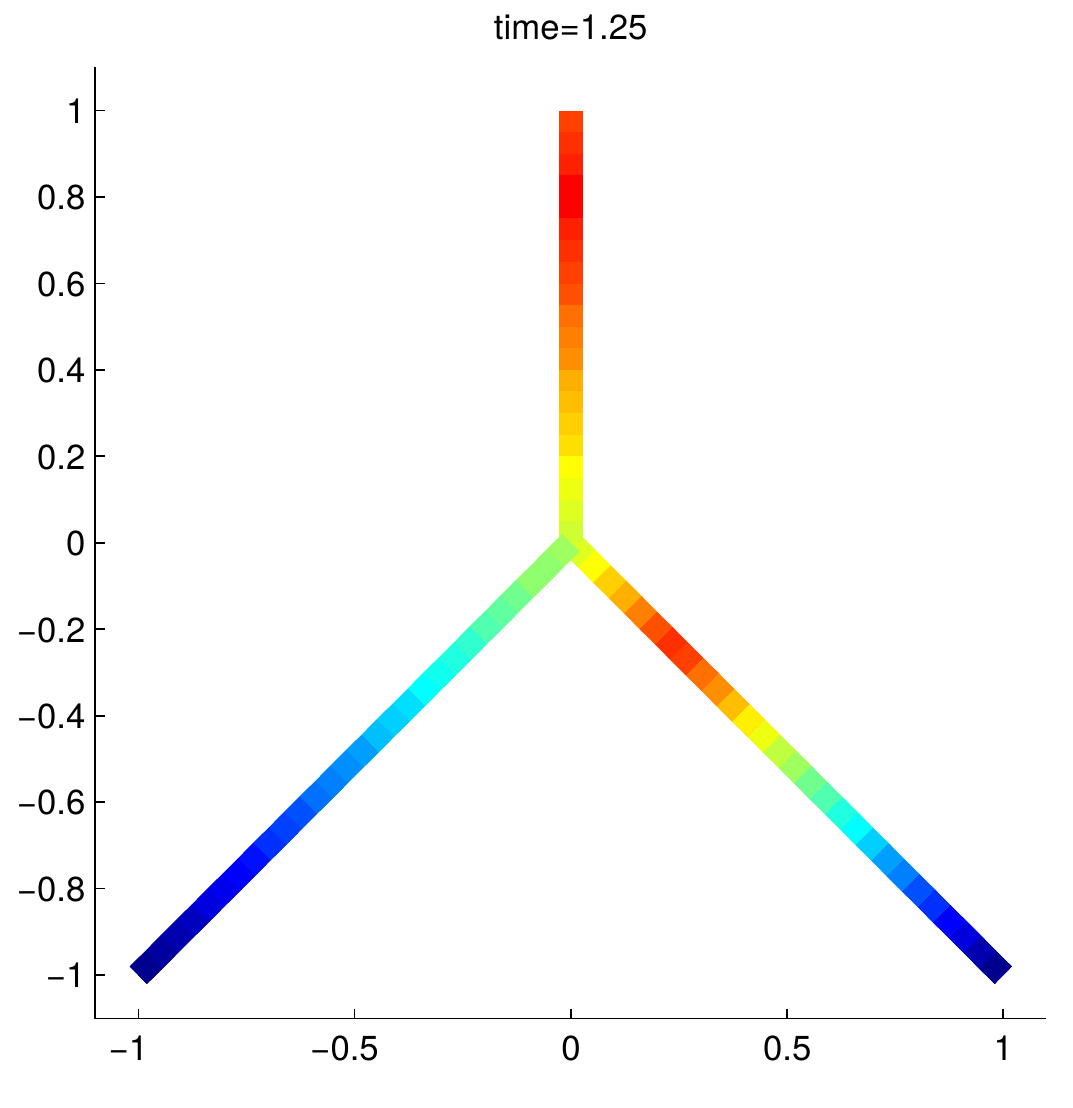}
\includegraphics[height=5.2cm]{colorbar.pdf}
\caption{Projection on the state coordinate plane of the  Initial Condition  (top left), numerical solution at time $t=0.4$ (top right), $t=0.8$ (bottom left) and $t=1.25$ (bottom right).} \label{fig1}
\end{center}
\end{figure}
We denote by $J_1$  the edge connecting $(0,1)$ to $(0,0)$ and by $J_2$, $J_3$ the edges connecting $(0,0)$ to $(1,-1)$ and $(-1,-1)$, respectively.
The cost function $L_i, i=0,
1,2,3$ are defined as follows
$$L_i(\alpha_i)=\left\{
\begin{array}{ll}
\frac{\alpha_i^2}{2}+1 &\hbox{  if }i=1,3,\\
\frac{\alpha_i^2}{2}+2 &\hbox{  if }i=2,0.
\end{array} \right.
$$
We impose  Dirichlet boundary conditions on the boundary nodes:
$$u(t,x)=\left\{
\begin{array}{ll}
0 &\hbox{  if }x=\{(-1,-1),(1,-1)\}\\
\sqrt{2}+1 &\hbox{  if }x=\{(0,1)\}.
\end{array} \right.
$$
The initial value $u_0$ is chosen as  the restriction of $1+x_2$ on $J$, where we denote $(x_1, x_2)=x$. In Figure \ref{fig1},  we show the color map of the  initial condition and of  the  numerical solution  at time $t=0.4,0.8,1.25$, projected on the state coordinate plane.  We can observe that the initial datum $u_0$ (Fig. \ref{fig1} left/top) quickly evolves to the stationary solution (Fig. \ref{fig1} right/bottom), which represents  a weighted distance from the boundary points, with exit costs equal to the Dirichlet boundary conditions.

 We compare the approximate solution at $T=2$ with the exact solution of the corresponding stationary problem; this makes sense since the approximate solution has already reached the steady state at time $T=2$.
 The exact steady state solution is
\begin{equation}\label{eqex}
u(x)=\left\{
\begin{array}{ll}
\sqrt{2}+x_2, &\hbox{  if }x \in J_1,\\
\min\left(2\sqrt{(x_1-1)^2+(x_2+1)^2},\sqrt{2}+2\sqrt{x_1^2+x_2^2}\right), &\hbox{  if }x\in J_2,\\
\sqrt{(x_1+1)^2+(x_2+1)^2}, &\hbox{  if }x\in J_3.
\end{array} \right.
\end{equation}
 In Figure \ref{fig2}, we show the behavior of the error \eqref{err} for various values of $\Delta$, fixing  the ratio between the spatial and the time step as $\Delta t=2.5\Delta x$. We underline that this is possible thanks  to the  stability property  of SL methods for large time steps (i.e. the classical hyperbolic CFL condition \cite{falcone2014semi} may not be  verified). We observe as in the first test a  linear decay of the $E^{\Delta}_\infty$ error. 

\medskip
\begin{figure}[t!]
\begin{center}
\includegraphics[height=6cm]{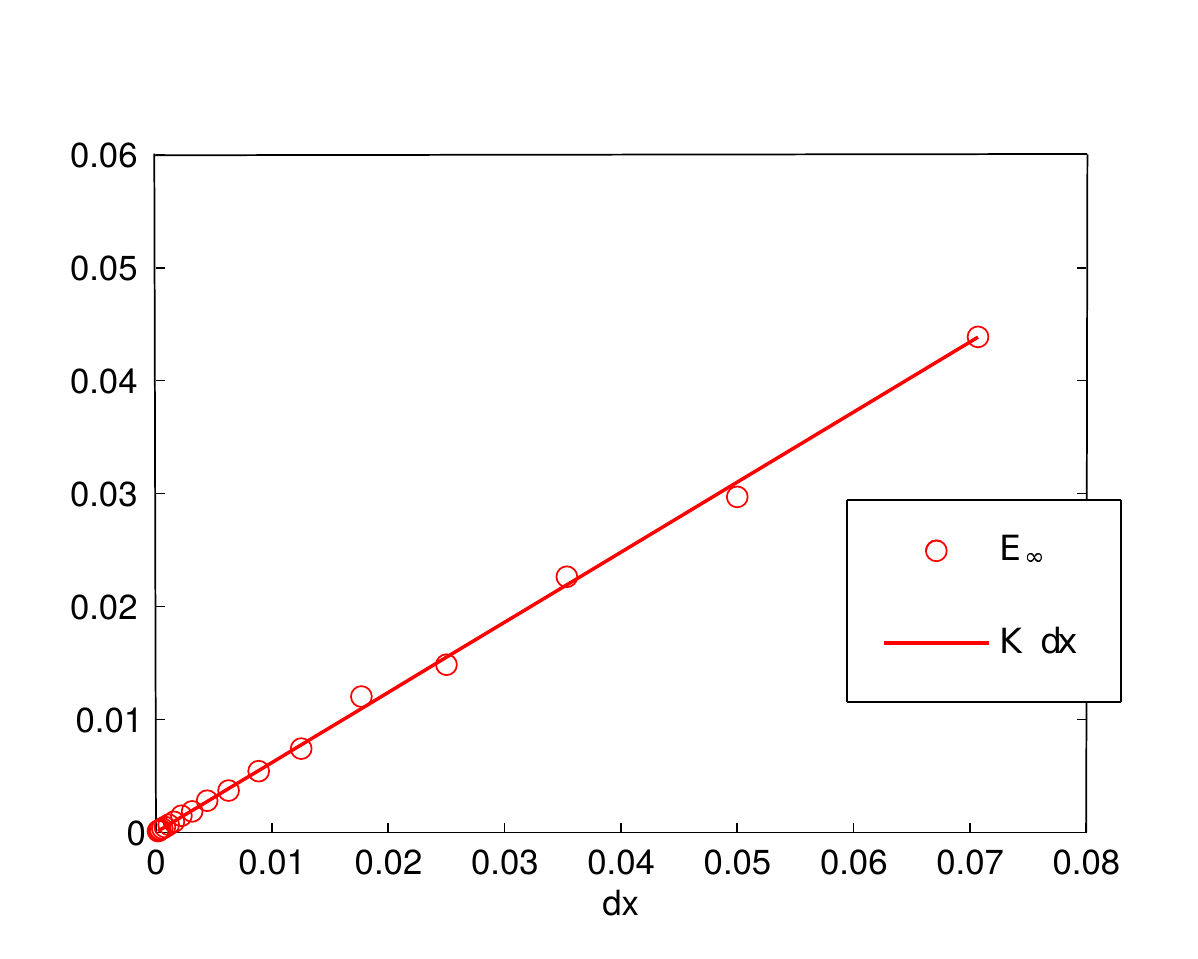} 
\caption{Graphic of $E_\infty^{\Delta}$($\circ$) with respect the space step, together with the line $K \Delta x$ with   $K=6.5$.} \label{fig2}
\end{center}
\end{figure}

\medskip
\paragraph{{\bf Test 3}} We conclude this section with a more realistic test where multiple edges are composing a complex traffic network. We consider the main network of the city of Rouen (Figure \ref{figR}, above)  and after some simplifications, we arrive at the network represented in Figure \ref{figR}, below. Here the edges in continuous blue line are large capacity roads, and in dashed red are smaller roads. The network is contained in the planar set $[0,1200]\times[0,2100]\subset \R^2$, that corresponds to the pixels of the reference map from which we extracted the network. For practical purposes, we scale it in the domain $[0,1]^2$.

\begin{figure}[ht!]
\begin{center}
\includegraphics[height=5.6cm]{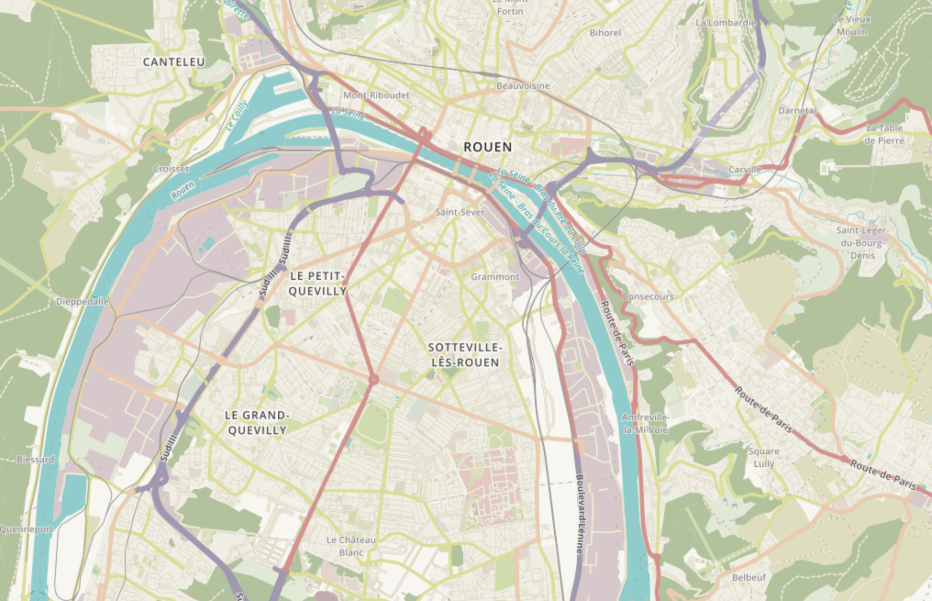}\\ 
\vspace{0.4cm}
\includegraphics[height=6cm]{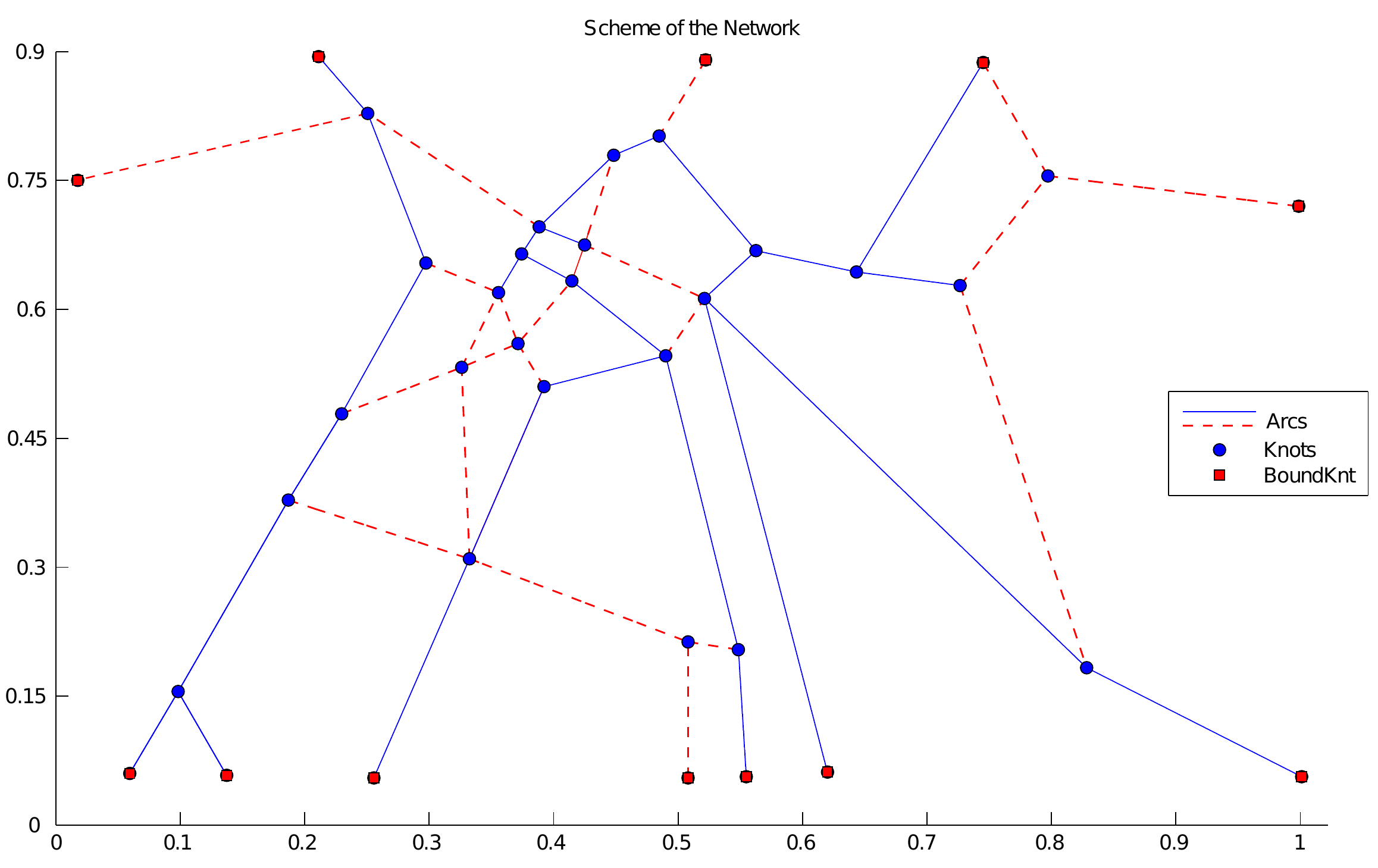} 
\caption{Map of the city of Rouen and simplified network modeling the structure of the road networks. In blue/solid line the bigger roads, in red/dotted line the smaller roads. The map $\copyright$ OpenStreetMap contributors. } \label{figR}
\end{center}
\end{figure}
Using the traffic flow interpretation of an HJ equation, as sketched in Section \ref{Sect:traffic}, we choose the initial datum $ u(x,0)=v(x)$ where $v(x)$ is the solution of the following stationary HJ equation
$$ \left\{\begin{array}{ll}
|v_x(x)|=0.7-\frac{(x-(0.5,0.5))^2}{2}, \quad & x\in J\\
v(x)=0, \quad & x\in \B.
\end{array}\right.$$
This case can be viewed as a special case of the stationary state of \eqref{eq:hjnet}, where the analysis is simpler since the Hamiltonian is continuous (cf. \cite{camilli2013approximation} for a detailed presentation). This choice of the initial datum models higher concentration of vehicles in correspondence of the city center (center of the domain). 

We are interested in the evolution of the \emph{density} of the vehicles $\rho$. We could derive it using $\eqref{cum}$. This procedure may be nontrivial (cf. \cite{costeseque2015convergent,imbert2013hamilton}) on the junctions and it is still a point that deserves further investigation. Instead, we adopt a numerical heuristic procedure using the relation $\rho(x,t)=-u_x(x,t)$ along every edge and defining 
$$\rho(x,t)=-\sum_i \min(\partial_i u(x,t),0), \quad \hbox{if } x\in \V$$ where the spatial derivatives are approximated thorough standard finite differences.  The numerical test that we present confirms that this procedure provides reasonable results.

We want to study the network in a case of an evacuation. We impose null Dirichlet boundary conditions on the exits of the network in correspondence of the red squares of Figure \ref{figR} (below).
We adopt the simple Hamiltonian
$$H_i(p)=\left\{
\begin{array}{ll}
\frac{1}{\lambda_i}p^2-p &\hbox{  if $p\geq 0$ },\\
\frac{1}{\lambda_i}p^2+p &\hbox{  if $p<0$,}
\end{array} \right.
$$
where $\lambda_i$ is the capacity of the arc $i$ namely $\lambda_i=4/5$ in the red edges and $\lambda_i=1$ elsewhere.


\begin{figure}[t!]
\begin{center}
\includegraphics[width=5.8cm,height=4cm]{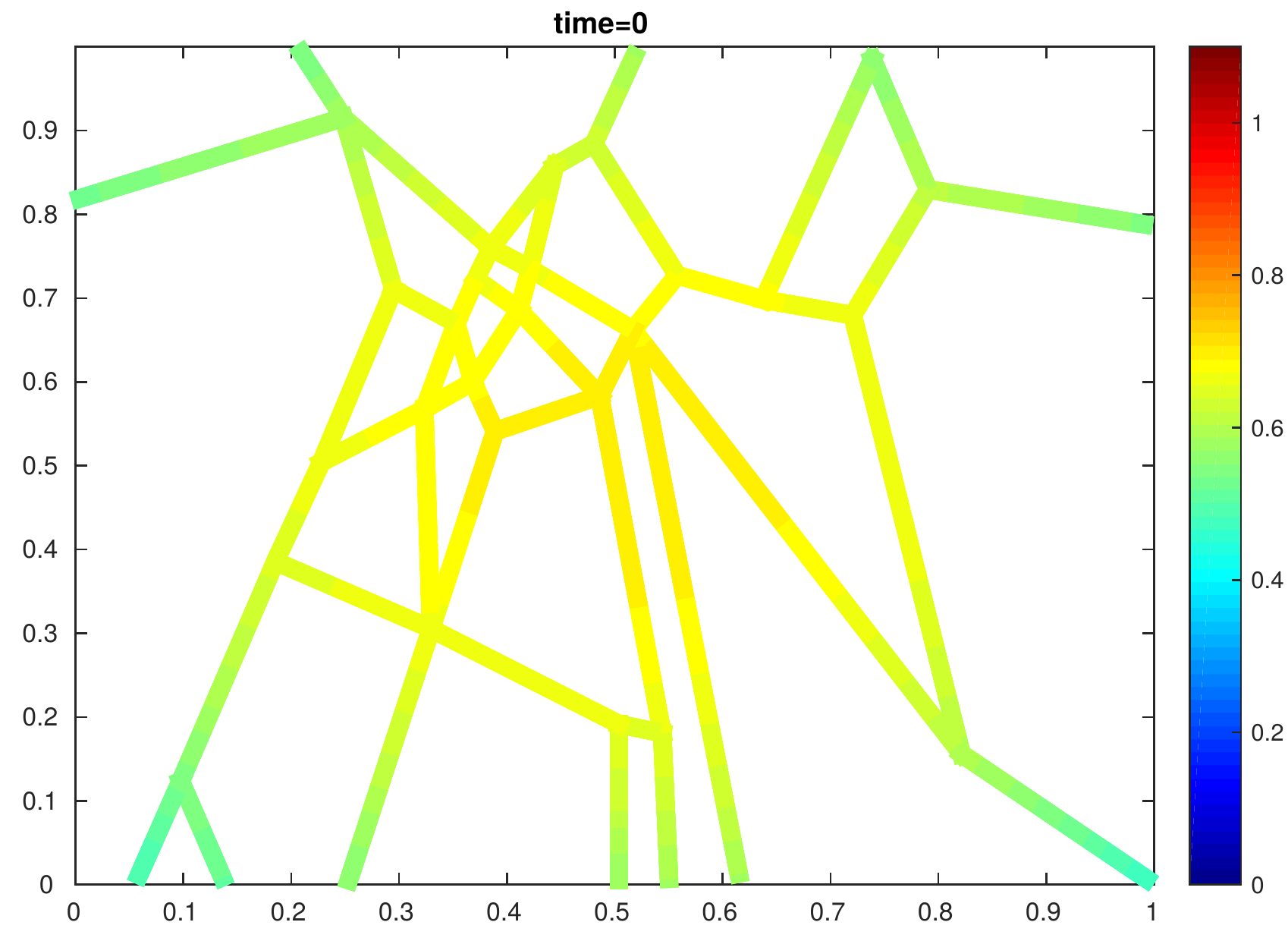}
\includegraphics[width=5.8cm,height=4cm]{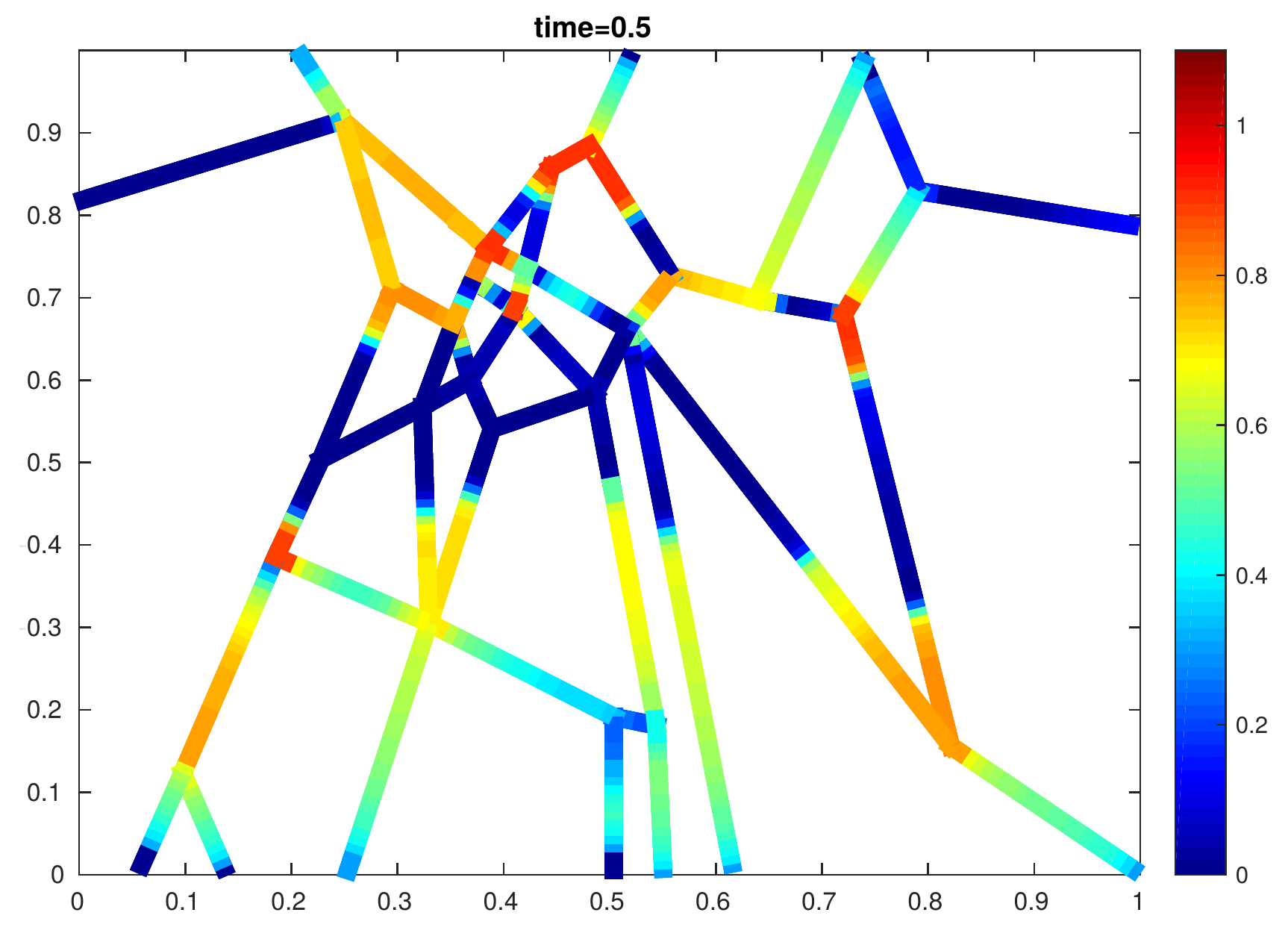} 
\includegraphics[width=5.8cm,height=4cm]{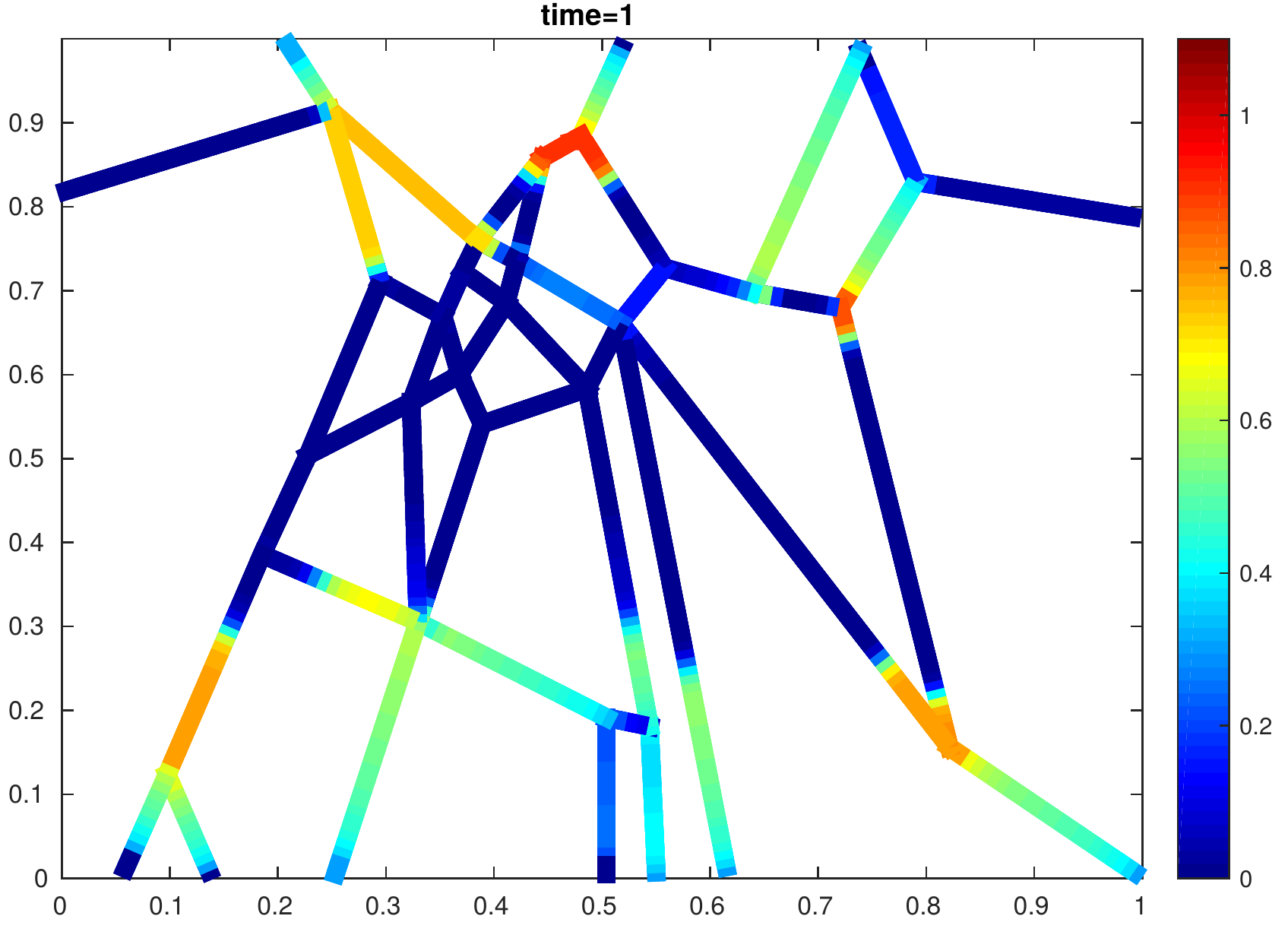} 
\includegraphics[width=5.8cm,height=4cm]{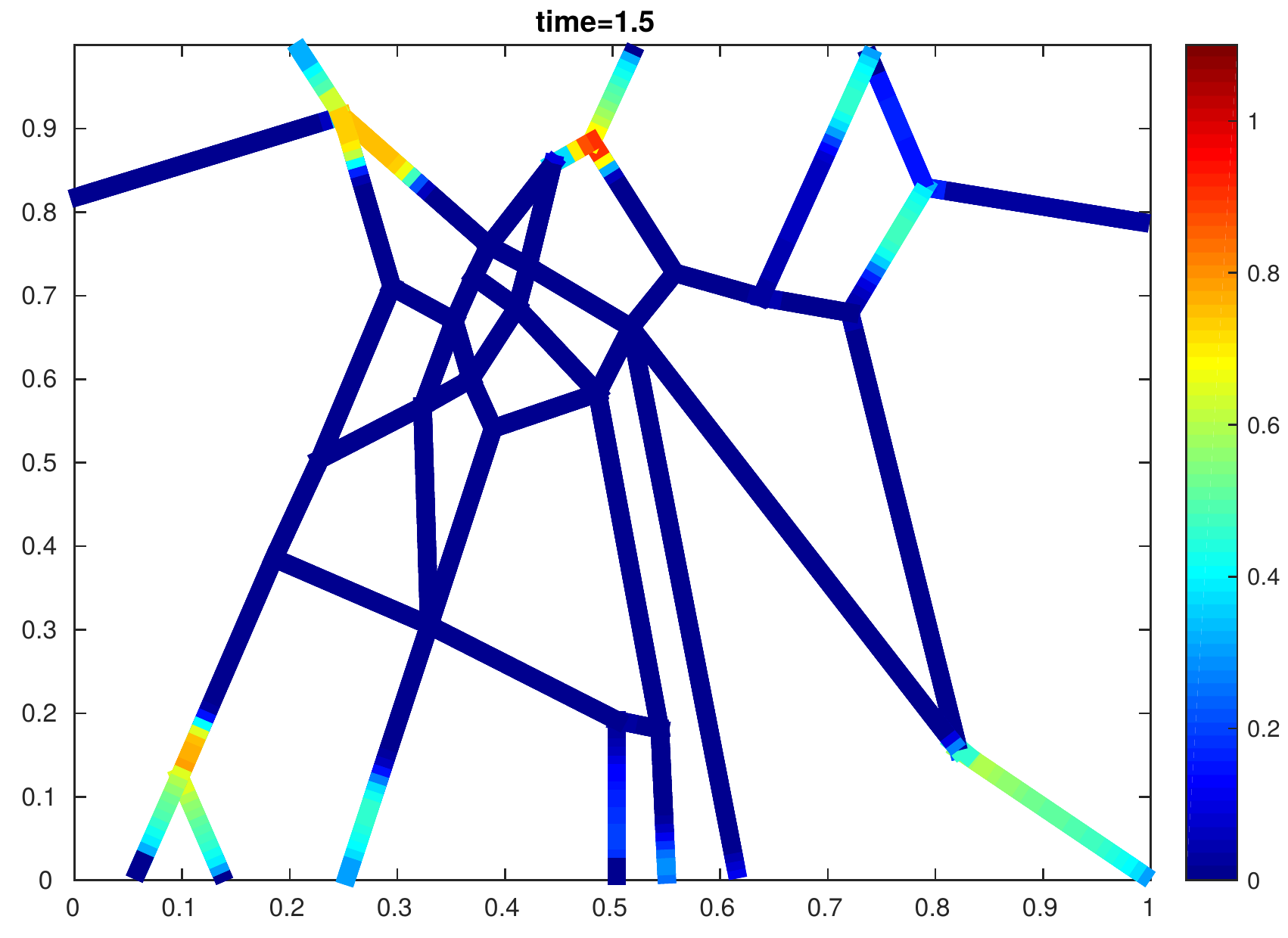} 
\caption{Evolution of the approximation of the density of vehicles at time $t=0,0.5, 1, 1.5$.} \label{figR2}
\end{center}
\end{figure} 
 
We observe that it is possible to obtain these Hamiltonians following \cite{imbert2013hamilton} for the classic LWR model (cf. \cite{garavello2016models}) with flux $(1-\rho)\rho$ on the blue arcs and $(1-5/4 \rho)\rho$ for the choice of $A=-0.4$. The red edges have then a reduced capacity compared to the blue ones.

We uniformly  approximate  the arcs using the discretization step $\Dx=0.01$ and we sample the time with $\Dt=0.05$. Due to the stability properties of the SL scheme no  CFL condition is needed, and we can adopt large time steps, fundamental to approximate the behavior of the system for long time scenarios.

In Figure \ref{figR2}, we can see the initial distribution of the density and its evolution in various moments. We observe the following. First of all the density, starting from a smooth configuration (the initial data is for construction $\rho(t,x)=0.7-0.5(x-(0.5,0.5))^2$) rapidly concentrates reaching some areas of maximal density. Those congested areas typically appear before a junction point. It is an intrinsic characteristic already observed in traffic flows literature. This phenomenon can only become more evident in the case of merging bifurcations where the outgoing roads have a reduced capacity. This is the case of the junction in the proximity of the point $(0.5, 0.82)$ where some congested areas are formed and take more time to disappear.

\section{Acknowledgments}
{The first author was supported by the Indam GNCS project ``Metodi numerici per equazioni iperboliche e cinetiche e applicazioni'' and PGMO project VarPDEMFG.
The second and third authors were partially supported by  the European Union with the European regional development fund (ERDF, HN0002137) and by the Normandie Regional Council (via the M2NUM project).}

\bibliographystyle{amsplain}
\bibliography{hjnetworks}

\end{document}

%% file: mymacro.tex
 \if {
  
 } \fi

\if {

}{{\caption}}
} \fi

\def\at{\tilde{a}}

\def\tt{\tilde{t}}

\def\B{\mathcal{B}}

\def\E{\mathcal{E}}

\def\G{\mathcal{G}}

\def\L{\mathcal{L}}

\def\N{\mathcal{N}}

\def\V{\mathcal{V}}

\def\eps{\varepsilon}

\newcommand{\floor}[1]{\lfloor #1 \rfloor}

\def\1B{{\bf  1}}

\newcommand{\NN}{\mathbb{N}}

\def\I{\mathbb{I}}

\newcommand\be{\begin{equation}}
\newcommand\ee{\end{equation}}
\newcommand\ba{\begin{array}}
\newcommand\ea{\end{array}}
\newcommand{\bean}{\begin{eqnarray*}}
\newcommand{\eean}{\end{eqnarray*}}

